\documentclass[10pt]{article}
\RequirePackage{amsthm,amsmath,amsfonts,amssymb}
\RequirePackage[numbers]{natbib}

\RequirePackage[colorlinks,allcolors=blue]{hyperref}

\RequirePackage{graphicx}

\usepackage[utf8]{inputenc}
\usepackage{bbm,mathabx,cases,booktabs}
\usepackage{enumitem}   
\usepackage{mathrsfs}
\usepackage{multirow}
\usepackage{hyphenat}
\usepackage{authblk}
\usepackage{natbib}

\makeatletter
\def\keywords{\xdef\@thefnmark{}\@footnotetext}
\makeatother

\usepackage[margin=1.35in]{geometry}
\usepackage{titlesec}
\titleformat{\subsubsection}[runin]
       {\normalfont\bfseries}
       {\thesubsubsection}
       {0.5em}
       {}
       [.]
       
\theoremstyle{plain}

\newtheorem{theorem}{Theorem}[section]
\newtheorem{lemma}[theorem]{Lemma}
\newtheorem{proposition}[theorem]{Proposition}
\theoremstyle{remark}
\newtheorem{remark}[theorem]{Remark}
\newtheorem{definition}[theorem]{Definition}
\newtheorem*{example}{Example}

\newcommand\nnfootnote[1]{%
  \begin{NoHyper}
  \renewcommand\thefootnote{}\footnote{#1}%
  \addtocounter{footnote}{-1}%
  \end{NoHyper}
}

\newcommand{\calD}{\mathcal{D}}
\newcommand{\calA}{\mathcal{A}}
\newcommand{\calF}{\mathcal{F}}
\newcommand{\calH}{\mathcal{H}}
\newcommand{\scrD}{\mathscr{D}}
\newcommand{\scrL}{\mathscr{L}}

\newcommand{\subscript}[2]{$#1 _ #2$}

\begin{document}


\title{\vspace{-1em}Characterization of the second order random fields subject to linear distributional PDE constraints}
\date{}



\author[1]{Iain Henderson \thanks{henderso@insa-toulouse.fr}}
\author[1]{Pascal Noble}
\author[1]{Olivier Roustant}
\affil[1]{Institut de Mathématiques de Toulouse; UMR 5219\\ 
Université de Toulouse; CNRS \\
INSA, F-31077 Toulouse, France}


\maketitle

\begin{abstract}
Let $L$ be a linear differential operator acting on functions defined over an open set $\calD \subset \mathbb{R}^d$. In this article, we characterize the measurable second order random fields $U = (U(x))_{x \in \calD}$ whose sample paths all verify the partial differential equation (PDE) $L(u) = 0$, solely in terms of their first two moments. When compared to previous similar results, the novelty lies in that the equality $L(u) = 0$ is understood in the \textit{sense of distributions}, which is a powerful functional analysis framework mostly designed to study linear PDEs. This framework enables to reduce to the minimum the required differentiability assumptions over the first two moments of $(U(x))_{x\in\calD}$ as well as over its sample paths in order to make sense of the PDE $L(U_{\omega})=0$. In view of Gaussian process regression (GPR) applications, we show that when $(U(x))_{x\in\calD}$ is a Gaussian process (GP), the sample paths of $(U(x))_{x\in\calD}$ conditioned on pointwise observations still verify the constraint $L(u)=0$ in the distributional sense. We finish by deriving a simple but instructive example, a GP model for the 3D linear wave equation, for which our theorem is applicable and where the previous results from the literature do not apply in general.
\end{abstract}

\nnfootnote{\emph{Key words} Generalized functions, Linear constraints, Linear Partial Differential Equations, Second order random fields}


\maketitle

\section{Introduction}

When dealing with an unknown function of interest $u : \calD \rightarrow \mathbb{R}$ where say $\calD \subset \mathbb{R}^d$, it is common (as e.g. in Bayesian inverse problems) to assume that it is a sample path of a random field $U=(U(x))_{x\in\calD}$. Incorporating prior knowledge over $u$, such as smoothness, is then achieved by constraining the law of $U$ accordingly.
Sometimes, this prior knowledge comes from physical considerations. If $u$ describes a positive quantity such as mass or energy, then the random variables $U(x)$ should all be positive almost surely (a.s.). In many cases, this physical constraint can be more precisely translated as a partial differential equation (PDE). Such equations are a pivotal tool for modelling, understanding and predicting real-life phenomena such as those arising from fluid mechanics, electromagnetics or biology to name a few. The most simple (yet central) PDEs are those that are linear. In this article, we will only consider homogeneous linear PDEs, which take the form
\begin{equation}\label{eq:edp}
L(u) := \sum_{|\alpha|\leq n}a_{\alpha}(x)\partial^{\alpha}u = 0.
\end{equation}
Above, $u$ is the unknown function of interest, defined over an open set $\calD\subset\mathbb{R}^d$, and $L$ is a linear partial differential operator. In (\ref{eq:edp}), for a multi-index $\alpha =(\alpha_1,...,\alpha_d)^T \in \mathbb{N}^d$, we used the notations $|\alpha| = \alpha_1+...+\alpha_d$ and $\partial^{\alpha} = (\partial_{x_1})^{\alpha_1}... (\partial_{x_d})^{\alpha_d}$. Homogeneous PDEs, i.e. PDEs with a null term on the right-hand side of \eqref{eq:edp}, are often encountered to describe conservation laws, such as conservation of mass, energy or momentum in closed systems \cite{Serre1999SystemsOC}.

In order to incorporate the knowledge that $L(u)=0$ in the prior $U$, a natural question is whether one can characterize, in terms of their law, the random fields $U$ whose sample paths are all solutions to the PDE \eqref{eq:edp}. Let $U$ be a centered second order random field with covariance function $k$: under the assumption that $U$ is a Gaussian process (GP) whose sample paths are $n$ times differentiable, \cite{ginsbourger2016} proved for some classes of differential operators $L$ of order $n$ that (\cite{ginsbourger2016}, Sections 3.3 and 4.1)
\begin{align}\label{eq:degeneracy}
\mathbb{P}(L(U) = 0)=1 \iff \forall x \in \calD, L(k(x,\cdot)) = 0.
\end{align}
This property provides a simple characterization of the GPs that incorporate the PDE constraint \eqref{eq:edp} sample path-wise. Such GPs would fall in the category of ``physics-informed'' GPs in the machine learning community. In the proof of this property, the fact that the sample paths are $n$ times differentiable, i.e. that the PDE \eqref{eq:edp} can be understood \textit{pointwise}, is central. These functions are then \textit{strong} solutions of the PDE \eqref{eq:edp} (see Definition \ref{def:strong_sol}).

In the standard PDE approach though, equation (\ref{eq:edp}) is reinterpreted by weakening the definition of the derivatives of $u$, thereby weakening the required regularity assumptions over $u$. It can indeed happen in practice that the sought solutions of the PDE $L(u)=0$ are not $n$ times differentiable or even continuous (see e.g. \cite{evans1998}, Section 2.1), and they are only solutions of some weakened formulation of equation (\ref{eq:edp}). This is typically the case for hyperbolic PDEs such as the wave equation presented in Section \ref{Section:GP_wave}. We introduce here the distributional formulation of the PDE (\ref{eq:edp}), where the regularity assumptions over $u$ are relaxed to the maximum. As such, this formulation enables working with potentially singular solutions of equation \eqref{eq:edp}, solutions which are not allowed to appear in more restrictive functional frameworks (see also the upcoming Remark \ref{rk:meas_val_sol}). Another advantage of the distributional formulation is that it provides a unifying framework for dealing with linear PDEs, independently of their nature. In contrast, traditional weak or variational formulations vary greatly depending on the nature of the PDE. As an illustration, in \cite{evans1998}, one can compare the different function spaces for weak solutions associated to elliptic PDEs (Section 6.1.2), parabolic PDEs (Section 7.1.1(b)) and hyperbolic PDEs (Section 7.2.1(b)). The distributional formulation will be our main object of interest in this article, and can be seen as a weakened form of weak formulations of PDEs.
Consider equation (\ref{eq:edp}),  and ``test it locally": that is, multiply it by a compactly supported, infinitely differentiable test function $\varphi$ (i.e. $\varphi \in C_c^{\infty}(\calD)$) and integrate over $\calD$:
\begin{equation}\label{eq:towards_distribs}
\forall \varphi \in C_c^{\infty}(\calD),\ \sum_{|\alpha|\leq n}\int_{\calD} \varphi(x)a_{\alpha}(x)\partial^{\alpha}u(x)dx = 0.
\end{equation}
For each integral term above, perform $|\alpha|$ successive integrations by parts to transfer the derivatives from $u$ to $\varphi$. Since $\varphi$ is identically null on a neighbourhood of the boundary of $\calD$, the boundary terms of each integration by parts vanish and we obtain that
\begin{equation}\label{eq:def_sol_distribs}
\forall \varphi \in C_c^{\infty}(\calD),\ \int_{\calD} u(x)\sum_{|\alpha|\leq n}(-1)^{|\alpha|}\partial^{\alpha}(a_{\alpha}\varphi)(x)dx = 0.
\end{equation}
To make sense of \eqref{eq:def_sol_distribs}, one only requires $u$ to be locally integrable, i.e. $\int_K|u(x)|dx < +\infty$ for all compact set $K\subset \calD$. We then say that a locally integrable function $u$ is a solution to $L(u)=0$ in the sense of distributions, or distributional sense, if $u$ verifies (\ref{eq:def_sol_distribs}). In this case, $u$ is a solution to equation \eqref{eq:edp} in the sense of ``all smooth local averages'' (i.e. for all $\varphi\in C_c^{\infty}(\calD)$), though not pointwise in general: taking $\varphi(x) = \delta_0(x-x_0)$ is not allowed without additional assumptions over $u$.

The distributional formulation of the PDE $L(u)=0$ is ``compliant with physics'' too, as pointed out by W. Rudin (\cite{rudin1991}, p. 150): most of the sensors we use in practice are only capable of computing local averages of the physical quantity they are measuring. Suppose one wishes to check experimentally that a temperature field obeys the heat equation, by using a set of thermometers: then one will actually only deal with the distributional formulation of the heat equation.

The natural question that follows from this new definition is whether one can characterize, in terms of their law, the random fields whose sample paths are solutions to the PDE $L(u)=0$ \textit{in the distributional sense}. The answer is yes, and is the main content of this article. Under the assumptions that $U$ is a measurable centered second order random field and that its standard deviation function ${\sigma : x\longmapsto \sqrt{k(x,x)}}$ is locally integrable, we show in Proposition \ref{prop :diff_constraints_distrib} that
\begin{equation}\label{eq:degeneracy_distrib}
\mathbb{P}(L(U) = 0 \ \textrm{in the distrib. sense})=1 \iff \forall x \in \calD, L(k(x,\cdot)) = 0 \ \textrm{in the distrib. sense}.
\end{equation}

\subsection*{Related literature}
It is known, at least since the fifties, that some covariance functions are naturally linked to certain stochastic partial derivative equations (SPDEs), i.e. PDEs where the source term is random. For example, it was already observed in 1954 by \cite{whittle1954stationary} that the covariance function of a stationary GP $U$ verifying the two dimensional SPDE $(\alpha^2 - \partial_{xx}^2 - \partial_{yy}^2)^{3/4}U = W_S$, where $W_S$ is a spatial white noise process, is exponential, i.e. of the form $\text{Cov}(U(x+h),U(x))= C\exp(-\alpha|h|)$. Already for this SPDE, the differentiation has to be understood in a weakened sense as white noise processes are not random fields in the usual sense. In \cite{roques2022spatial}, SPDEs describing the random motion of micro-particles are introduced to link certain covariance functions, Matérn in particular, with an underlying physical model. We also refer to \cite{lindgren2022spde} for a large overview of the possible applications and recent developments pertaining to random fields defined by SPDEs. A general framework for the study of SPDEs was recently reintroduced in \cite{vergara2022general}, which was then used to
 classify the stationary \textit{generalized} random fields that are solutions of a wide class of linear SPDEs. In particular, \cite{vergara2022general} provides a description of all the second order stationary generalized random fields that are solutions to certain homogeneous PDEs, and the 3D wave equation in particular (which we also study in Section \ref{Section:GP_wave}), in terms of their covariance operator. Loosely speaking, generalized random fields are function-indexed random fields where the covariance function is replaced by a covariance operator. From a functional analysis point of view, this is actually very close to the tools we use here, although in this article we constrain ourselves to work with (standard) random fields with well-defined sample paths, as these are the objects that arise the most in  the random function models met in practice. The two other key differences between this work and \cite{vergara2022general} are that $(i)$ we do not focus on \textit{stationary} random field models for $u$ and $(ii)$ we focus on the \textit{homogeneous} case for PDE $L(u)=0$.

The literature concerning random fields that are PDE-constrained at the level of the sample paths is rather sparse. In \cite{ginsbourger2016}, general theorems are exposed for many different classes of linear operators acting on  suitable spaces of functions. These theorems take the form of equation \eqref{eq:degeneracy}, and can in turn be applied to certain differential operators (see \cite{ginsbourger2016}, Sections 3.3 and 4.1). \cite{scheuerer2012} builds covariance functions that ensure that the sample path of a given two or three dimensional random vector field are either divergence or curl free. This result is notable because  ``any'' three dimensional vector field can be decomposed as a sum of divergence and curl free vector fields through the Helmholtz-Hodge decomposition theorem. Moreover, divergence or curl free vector fields are commonly encountered in fluid mechanics. \cite{fiedler2016distances} extends the results of \cite{scheuerer2012} to random fields on the sphere of $\mathbb{R}^3$, which has been rediscovered later in \cite{fan2018modeling}. In \cite{estrade2020anisotropic}, stationary GPs are represented in terms of a random wavevector. \cite{estrade2020anisotropic} then characterizes the stationary GPs whose sample paths verify a homogeneous linear PDE, in terms of the spectral measure of the GP and in terms of its random wavevector. \cite{estrade2020anisotropic} additionally requires that the sample paths be infinitely differentiable, that the PDE's coefficients be constant and that only even orders of differentiation appear in the PDE. This is then applied to a few wave models. A simple algorithm for building linearly constrained GPs is proposed in \cite{jidling2017}, based on formal GPR derivations upon \eqref{eq:edp}; however, partly because the assumed regularity of $u$ is not fully addressed, the claim that the sample paths of the underlying GP are indeed linearly constrained is left unproved. This is clarified in \cite{hegerman2018}, where the requirement that $u \in C^{\infty}(\mathbb{R}^d)$ is made explicit and the enforcement of the PDE on the sample paths is proved for GPs whose sample paths are smooth. The algorithm from \cite{jidling2017} is then supplemented in \cite{hegerman2018}, where parametrizations of the solution spaces of \eqref{eq:edp} thanks to Gr\"{o}bner bases are proposed. In \cite{hegermann2021LinearlyCG}, the same author completes the approach from \cite{hegerman2018} by incorporating boundary conditions on hypersurfaces in the Gr\"{o}bner basis parametrization. With the idea to apply GPR to rigid body dynamics, \cite{geist2020} enforces Gauss' principle of least constraint on the sample paths of a GP. 

One can understand our main result (Proposition \ref{prop :diff_constraints_distrib}) as a characterization of the ``physics-informed'' random fields that incorporate the distributional PDE constraint $L(u)=0$ at the level of the sample paths. It turns out that the design of similar ``physics-informed priors'' has received a lot a attention from the machine learning community since the early 2000' (\cite{graepel2003}), in the context of Gaussian process regression (GPR); see Section \ref{subsub:GPR} for a description of this technique. GPR is a Bayesian framework for function regression and interpolation which is well suited for handling linear constraints, partly because GPs are ``stable under linear combinations'', see Section \ref{subsub:second_GP}. The recurring idea is to assume that the function $u$ in equation \eqref{eq:edp} is a sample path of a (centered) GP $U$ and to draw the consequences of equation \eqref{eq:edp} on the covariance function of $U$. The covariance function of $U$ is then expected to incorporate the constraint $L(u)=0$ in some sense. The majority of these works (except those mentioned above) do not aim at analysing whether the obtained covariance function indeed yields sample path PDE constraints over $U$: this is justified by the fact that they are only concerned with imposing the constraints on the function provided by GPR to approximate $u$. This approximation of $u$, which we denote by $\tilde{m}$, is called the Kriging mean in the GPR context; see equation \eqref{eq:krig mean} for a definition.

While they do not primarily focus on investigating sample path PDE constraints (contrarily to this article), the works coming from the GPR community are still very connected to this article. Indeed, they are concerned with designing explicit covariance functions that verify constraints of the form $L(k(\cdot,x))=0$ for all $x\in\calD$ (the PDE is understood in the strong sense in these works). Indeed, this constraint ensures that all the possible regression functions $\tilde{m}$ provided by the corresponding GPR model verify the constraint $L(\tilde{m})=0$ (as seen in equation \eqref{eq:krig mean}). Note that ``$L(k(\cdot,x))=0 \ \forall x\in\calD$'' is the right-hand side of equation \eqref{eq:degeneracy}: actually knowing covariance functions that verify this constraint is a necessary complement to the condition we prove in this article (Proposition \ref{prop :diff_constraints_distrib} is otherwise useless in practice). Explicit PDE constrained covariance functions were designed for a number of classical PDEs, namely: divergence-free vector fields \cite{Narcowich1994GeneralizedHI,scheuerer2012}, curl-free vector fields \cite{fuselier2007refined,scheuerer2012}, the Laplace equation \cite{Schaback2009SolvingTL,mendes2012,albert2020}, Maxwell's equations \cite{wahlstrom2013,jidling2017,hegerman2018} (although \cite{wahlstrom2013,jidling2017} only exploit curl/divergence free constraints), the 1D heat equation \cite{albert2020}, Helmholtz' 2D equations \cite{albert2020}, and linear solid mechanics \cite{Jidling2018ProbabilisticMA}.
\cite{hegermann2021LinearlyCG} and \cite{gulian2022} enforce homogeneous boundary conditions on the covariance function. 

We finish with a brief overview of the alternative ``physics-informed'' GPR models. Contrarily to the equation \eqref{eq:edp} considered here, one may put a random source term $f$ in the PDE and study instead the SPDE $L(u)=f$ : see \cite{raissi2017}, \cite{alvarez2013} and \cite{owhadi_bayes_homog} for entry points on the related literature. A recent article \cite{CHEN_owhadi_2021} extended the use of GPR to nonlinear PDEs by imposing the nonlinear interpolation constraints on the collocation points, setting the way forward for many possible applications of GPR to nonlinear realistic
PDE models, as found e.g. in fluid mechanics.
In \cite{NGUYEN_peraire}, the variational formulation (see \cite{evans1998}, Section 6.1.2 for a definition) of certain linear PDEs has been incorporated into a GPR framework. This approach requires the use of Gaussian generalized random fields (see \cite{agnan2004}, Section 2.2.1.1), or ``functional Gaussian processes'' following \cite{NGUYEN_peraire}. The variational formulation of a PDE differs from its distributional formulation in the choice of the space of test functions.
\subsubsection*{Contribution and organisation of the paper}
Consider the PDE in equation \eqref{eq:edp}, where the coefficients of the differential operator $L$ have possibly limited smoothness.
Consider also a centered second order measurable stochastic process ${U=(U(x))_{x\in\calD}}$ with covariance function $k(x,x')$ (see Sections \ref{subsub:sto_pro_mes} and \ref{subsub:second_GP}). Under the assumption that its standard deviation function {$\sigma : x\longmapsto k(x,x)^{1/2}$} is locally integrable, we show in Proposition \ref{prop :diff_constraints_distrib} that the announced equation \eqref{eq:degeneracy_distrib} holds.
The result is then compared to a previous result from \cite{ginsbourger2016}, which ensures pointwise linear differential degeneracy of the sample paths of $U$ under stronger assumptions. We then provide a simple corollary which states that linear distributional differential constraints are preserved when a GP $U$ is conditioned on pointwise observations, in view of GPR applications.

As an application example, we derive a general Gaussian process model for the homogeneous 3D free space wave equation, for which the solutions are not smooth in general. This equation is central for describing finite speed propagation phenomena as found e.g. in acoustics. Plugging this model in a GPR framework yields potential applications in different inverse problems related to this PDE, such as thermoacoustic tomography (i.e. initial condition reconstruction, \cite{kuchment_tomo}, Section 19.3.1.1), source localization or propagation speed estimation, following e.g. the GPR methodology from \cite{raissi2017} or \cite{ginsbourger2016}, Section 4.2.

This model is derived by putting GP priors over the initial conditions of the wave equation and in Proposition \ref{prop : wave kernel}, we obtain ``explicit'' formulas for the covariance function of the solution process, in the form of convolutions. From Propositions \ref{prop :diff_constraints_distrib} and \ref{prop : wave kernel}, we obtain that the sample paths of the corresponding (nonstationary) GP all verify the wave equation in the distributional sense. When the covariance functions of the initial conditions are not smooth enough, the result from \cite{ginsbourger2016} cannot be applied. Explicitly, for this PDE, choosing the commonly used 3/2-Mat\'ern covariance functions for the initial position is enough to land outside the scope of the result from \cite{ginsbourger2016} (Section \ref{subsub:extend_kwave}).

We emphasize that the covariance expressions exposed in Proposition \ref{prop : wave kernel} are original and interesting in themselves, as they can be used for efficient GPR for the wave equation.
Specifically, the key difference with the wave equation covariance functions presented in \cite{vergara2022general} is that here, no stationarity assumptions are made on the solution stochastic process $U$. In particular the spectral measure provided by Bochner's theorem \cite{gpml2006}, which is the key tool used in \cite{vergara2022general}, is not available anymore. We thus resort to more standard integration techniques to prove Proposition \ref{prop : wave kernel}.

The paper is organized as follow. For self-contain\-ment, Section \ref{Section:background} is dedicated to reminders on random fields and generalized functions. This Section and all the proofs are detailed enough so that this article is accessible to the analyst, the probability theorist and the statistician. In Section 3, we state and prove our new necessary and sufficient condition on random fields that are subject to linear distributional differential constraints. Section 4 is dedicated to the study of a GP model for the wave equation. We conclude in Section 5.

\section{Background}\label{Section:background}

\subsection{Random fields}\label{sub:sto_pro}
Let $(\Omega, \calA,\mathbb{P})$ be a probability space. For convenience, we will assume that it is complete, i.e. that $\calA$ contains the subsets of sets $A \in \calA$ such that $\mathbb{P}(A)=0$.

\subsubsection{Basic definitions}\label{subsub:def} Let $\calD \subset \mathbb{R}^d$ be an open set. In this article, a random field $U = (U(x))_{x \in \calD}$ is a collection of real random variables defined on $\Omega$. We define its sample path at point $\omega \in \Omega$ to be the deterministic function $x \longmapsto U(x)(\omega)$, and we denote it by $U_{\omega}$. Given an operator acting on the sample paths of $U$, an event of the form $\{L(U) \in A\}$ will always be understood sample path wise: that is, by definition, $\{L(U) \in A\} := \{\omega \in \Omega : L(U_{\omega}) \in A\}$. Such sets are not automatically measurable; still, they are measurable as soon as they contain an event of probability $1$ (as the ones in Propositions \ref{prop : lin constraints classic} and \ref{prop :diff_constraints_distrib}), since $(\Omega, \calA,\mathbb{P})$ is a complete probability space.

\subsubsection{Measurable random fields}\label{subsub:sto_pro_mes} In view of our main theorem, a necessary notion is that of the measurability of the random field $U$. $U$ is said to be measurable (\cite{doob_sto_pro}, p. 60 or \cite{legall2013}, p. 34) if it is measurable seen as a bivariate map $U : (\Omega \times \calD, \calA\otimes \mathcal{B}(\calD)) \longrightarrow (\mathbb{R},\mathcal{B}(\mathbb{R})),\  (\omega,x) \mapsto U(x)(\omega)$. Here, $\mathcal{B}(S)$ denotes the Borel $\sigma$-algebra of $S$ and $ \calA\otimes \mathcal{B}(\calD)$ denotes the product $\sigma$-algebra of $\calA$ and $\mathcal{B}(\calD)$. 

To work with measurable random fields, one will often consider random fields $U$ which are continuous \textit{in probability}, i.e. for all $x\in\calD$ and $\varepsilon > 0, \mathbb{P}(|U(x)-U(x+h)|>\varepsilon) \rightarrow 0$ when $h\rightarrow 0$. Indeed, continuity in probability implies the existence of a measurable modification of $U$, i.e. a measurable random field $\tilde{U}$ such that $\mathbb{P}(\tilde{U}(x)=U(x))=1$ for all $x\in\calD$ (\cite{doob_sto_pro}, Theorem 2.6, p. 61). One then implicitly works with $\tilde{U}$. In this article, we will directly assume that we deal with measurable random fields instead of assuming any continuity regularity on the sample paths of the said stochastic process. This is because pointwise continuity is not really relevant when working with PDEs in a weak sense; actually, one of the main points of working with weakened formulations is to avoid strong (i.e. pointwise) formulations. Note however that ensuring measurability outside of the above mentioned theorem, though possible, rapidly becomes tedious (see e.g. \cite{doob1937stochastic}, Theorem 2.3). A famous theorem from Kolmogorov (\cite{da2014stochastic}, Theorem 3.3 p. 73 and Theorem 3.4 p. 74) provides sufficient conditions for almost sure continuity of the sample paths, which in turn implies continuity in probability of the random field. This condition is phrased in terms of a sufficient Hölder control of the expectation of the increments of the process. Refinements in the case of Gaussian processes exist: see e.g. \cite{adler2007}, Theorem 1.4.1, p. 20. On a final note, the measurability assumption is discussed in \cite{steinwart2019convergence} (Theorem 3.3), where it is shown to be a necessary condition for the existence of Karhunen-Loève expansions of second order random fields.

\subsubsection{Second order random fields, Gaussian processes.}\label{subsub:second_GP} Note $L^2(\mathbb{P})$ the Hilbert space of real valued random variables $X$ such that $\mathbb{E}[X^2] < +\infty$. A stochastic process $(U(x))_{x\in\calD}$ is said to be second order if for all $x \in \calD, \ U(x)\in L^2(\mathbb{P})$. One can then define its mean and covariance functions by $m(x) = \mathbb{E}[U(x)]$ and $k(x,x')=\mathbb{E}[(U(x)-m(x))(U({x'})-m(x'))]$ respectively. One can then also define its standard deviation function
\begin{align}\label{eq:def_std_func}
\sigma : x \mapsto \sqrt{k(x,x)}.
\end{align}
A Gaussian process $(U(x))_{x\in\calD}$ over $\calD$ is a random field over $\calD$ such that for any $(x_1,...,x_n) \in \calD^n$ and any $(a_1,...,a_n)\in\mathbb{R}^n, \sum_i a_iU(x_i)$ is a Gaussian random variable; that is, the law of $(U({x_1}),...,U({x_n}))^T$ is a multivariate normal distribution.
The law of a GP is characterized by its mean and covariance functions (\cite{janson_1997}, Section 8). We write $(U(x))_{x \in \calD} \sim GP(m,k)$. Given a GP $(U(x))_{x\in\calD}$, we will sometimes use the space $\mathcal{L}(U) = \overline{\text{Span}(U(x), x \in \calD)}$, i.e. the Hilbert subspace of $L^2(\mathbb{P})$ induced by $U$. Since  $L^2(\mathbb{P})$-limits of Gaussian random variables drawn from the same GP remain Gaussian (\cite{janson_1997}, Section 1.3), $\mathcal{L}(U)$ only encompasses Gaussian random variables.

Whereas $m$ can be any function, the covariance function $k$ has to be symmetric and positive definite: for all $(x_1,...,x_n)$ in $\calD^n$, the matrix $(k(x_i,x_j))_{1\leq i,j \leq n}$ is symmetric and nonnegative definite.

Symmetric positive definite functions verify the Cauchy-Schwarz inequality \cite{gpml2006} :
\begin{align}\label{eq:CS PD}
\forall x,x' \in \calD, \ \ \ |k(x,x')| \leq \sqrt{k(x,x)}\sqrt{k(x',x')}.
\end{align} 
Note that there is a one-to-one correspondence between positive definite functions and the laws of centered GPs (\cite{doob_sto_pro}, Theorem 3.1). We provide below two examples of radial Matérn covariance functions (\cite{gpml2006}, pp. 84-85), which will be useful in Section \ref{Section:GP_wave}. Set $r = ||x-x'||$, the Euclidean distance between $x$ and $x'$, then the following two functions are valid covariance functions, given any $l > 0$:
\begin{align}\label{eq:matern}
k_{1/2}(x,x') &= \exp(-r/l),\ \ \ \ \ k_{3/2}(x,x') = (1+r/l)\exp(-r/l).
\end{align}

These covariance functions are widely used in machine learning, especially $k_{3/2}$. Almost surely, the sample paths of a GP with a Matérn covariance function $k_{\nu}$ with $\nu = m + 1/2,\ m\in \mathbb{N},$ are of differentiability class $C^m$ and not $C^{m+1}$. They are thus commonly used to model functions with finite smoothness. 

\subsection{Tools from functional analysis}\label{sub : distrib}\label{sub:func_analysis}
We refer to \cite{rudin1991} and \cite{treves2006topological} for further details on generalized functions and Radon measures. In this whole subsection, $\calD$ is an open set of $\mathbb{R}^d$.

\subsubsection{\texorpdfstring{Class $C^m$ functions, test functions, locally integrable functions}{functionspaces}}\label{subsub:L1loc_test} Given $m\in\mathbb{N}$, $C^m(\calD)$ denotes the space of real-valued functions defined over $\calD$ of class $C^m$, and $C_c^m(\calD)$ denotes the subspace of $C^m(\calD)$ of functions $\varphi$ whose support $\text{Supp}(\varphi)$ is compact. Recall that $\text{Supp}(\varphi)$ is the closure of the set $\{x : \varphi(x) \neq 0\}$.
The space $C_c^{\infty}(\calD)$, which we will rather denote $\scrD(\calD)$, is the space of compactly supported infinitely differentiable functions supported on $\calD$, also known as test functions. $L^1_{loc}(\calD)$ denotes the space of measurable scalar functions $f$ defined on $\calD$ that are locally integrable, i.e. such that $\int_K |f| < +\infty$ for all compact sets $K \subset \calD$. Two locally integrable functions are equal in $L^1_{loc}(\calD)$ when they are equal almost everywhere (a.e.) in the sense of the Lebesgue measure over $\mathbb{R}^d$. $L_{loc}^1(\calD)$ is a very large space which contains the space of piecewise continuous functions, but also all the local Lebesgue spaces $L_{loc}^p(\calD), p \geq 1$ and thus all the Sobolev spaces of nonnegative exponent. It is in fact the largest space of functions that can be alternatively viewed as continuous linear forms over $\scrD(\calD)$ (see Section \ref{subsub:regular_distrib} below).

\subsubsection{Generalized functions}\label{subsub:gen_func} We endow $\scrD(\calD)$ with its usual LF-space topology, defined for example in \cite{treves2006topological}, Chapter 13. LF stands for ``strict inductive limit of Fréchet spaces''. As it will appear in several places later on, we briefly describe the LF topology following \cite{treves2006topological}, although this is not necessary for understanding the article. Assume that a vector space $E$ can be written as $E = \bigcup_n E_n$ where $(E_n)$ is an increasing sequence of Fréchet spaces (i.e. metrizable complete locally convex topological vector spaces), such that the natural injection $E_n \rightarrow E_{n+1}$ is a linear homeomorphism over its range. The LF topology over $E$ is defined as follow: a convex set $V \subset E$ is a neighborhood of $0$ if and only if $V\cap E_n$ is a neighborhood of $0$ for all $n$. It is remarkable that LF topologies are \textit{not} metrizable except if for some $n_0, \ E_n = E_{n_0}$ for all $n\geq n_0$ (\cite{treves2006topological}, Remark 13.1). In return, this allows for some other very nice topological properties to hold, e.g., LF spaces are complete (\cite{treves2006topological}, Theorem 13.1).

For $\scrD(\calD)$, the LF topology is the one corresponding to the decomposition $\scrD(\calD) = \bigcup_i\scrD_{K_i}(\calD)$, where $\scrD_{K_i}(\calD):= \{\varphi \in C^{\infty}(\calD) : \text{Supp}(\varphi) \subset K_i\}$, and $(K_i)_{i \in \mathbb{N}}$ is an increasing sequence of compact subsets of $\calD$ such that $\bigcup_i K_i = \calD$ (\cite{treves2006topological}, pp. 131-133). This LF topology does not depend on the choice of $(K_i)_{i \in \mathbb{N}}$. An example of metric inducing the Fréchet topology of $\scrD_{K_i}(\calD)$ is the following:
\begin{align}\label{eq:def_di}
 d_i(\varphi,\psi) := \sup_{N \in \mathbb{N}}2^{-N} \frac{p_{N,i}(\varphi - \psi)}{1+p_{N,i}(\varphi - \psi)}, \ \ \ p_{N,i}(\varphi) := \max_{|\alpha|\leq N} \sup_{x\in K_i}|\partial^{\alpha}\varphi(x)|.
\end{align}
It is given in \cite{rudin1991}, Section 1.46 p. 34 and Remark 1.38$(c)$ p. 29. Explicitly, a sequence $(\varphi_n) \subset \scrD(\calD)$ converges to $\varphi \in \scrD(\calD)$ if there exists a compact set $K\subset \calD$ such that $\text{Supp}(\varphi_n) \subset K$ for all $n\in\mathbb{N}$ and for all $\alpha\in\mathbb{N}^d$, $||\partial^{\alpha}\varphi_n - \partial^{\alpha}\varphi||_{\infty} \rightarrow 0$ (\cite{rudin1991}, Theorem 6.5$(f)$ and the remark following p. 154).

We call generalized function any continuous linear form on $\scrD(\calD)$, i.e. any element of $\scrD(\calD)'$, the topological dual of $\scrD(\calD)$. We will rather denote it by $\scrD'(\calD)$ as in \cite{treves2006topological}, Notation 21.1. The topology of $\scrD(\calD)$ is such that $T \in \scrD'(\calD)$ if and only if for all compact set $K \subset \calD$, there exists $C_{K} > 0$ and a nonnegative integer $n_K$ such that
\begin{align}\label{eq:def distrib continue}
\forall \varphi \in \scrD(\calD)\ \  \text{such that} \ \text{Supp}(\varphi) \subset K,\ \  |T(\varphi)| \leq C_{K} \sum_{|\alpha|\leq n_K} ||\partial^{\alpha} \varphi||_{\infty}.
\end{align}
We recall that we use the following notations: for a multi-index $\alpha = (\alpha_1,...,\alpha_d) \in \mathbb{N}^d$, we denote $|\alpha| = \alpha_1+...+\alpha_d$ and $\partial^{\alpha} := (\partial_{x_1})^{\alpha_1}...(\partial_{x_d})^{\alpha_d}$ where $\partial_{x_i}^{\alpha_i}$ is the $\alpha_i^{th}$ derivative with reference to the $i^{th}$ coordinate $x_i$.
Generalized functions are also called ``distributions'', a terminology we will only use when there is no risk of confusion with probability distributions. The duality bracket will be denoted $\langle, \rangle$: for all $\varphi \in \scrD(\calD)$ and $T \in \scrD'(\calD)$, we have $\langle T, \varphi \rangle := T(\varphi)$. 

\subsubsection{Generalized functions and differentiation}\label{subsub:diff_of_distrib}
Any generalized function $T$ can be infinitely differentiated (\cite{rudin1991}, Section 6.12, p. 158 or \cite{treves2006topological}, pp. 248-250) according to the following definition
\begin{align}\label{eq:diff distrib}
\partial^{\alpha} T : \varphi \longmapsto \langle T, (-1)^{|\alpha|} \partial^{\alpha} \varphi \rangle.
\end{align}
The derivative $\partial^{\alpha} T$ is then also a continuous linear form over $\scrD(\calD)$, i.e. $\partial^{\alpha} T \in \scrD'(\calD)$.
\subsubsection{Regular generalized functions}\label{subsub:regular_distrib} Any function $f \in L^1_{loc}(\calD)$ can be injectively identified to a generalized function $T_f$ (\cite{treves2006topological}, p. 224 or \cite{rudin1991}, Section 6.11, p. 157) defined as follow
\begin{align}\label{eq:fL1 loc distrib}
\forall \varphi \in \scrD(\calD), \ \ \ \langle T_f, \varphi \rangle := \int_{\calD} f(x) \varphi(x) dx.
\end{align}
The map $L^1_{loc}(\calD) \ni f \longmapsto T_f$ is linear and injective; any generalized function $T$ that is of the form $T_f$ for some $f \in _{loc}^1(\calD)$ is said to be regular. Throughout this article, we will use the abusive notation $\langle T_f, \varphi \rangle = \langle f, \varphi \rangle$, as if $\langle, \rangle$ were the $L^2$ inner product.
Observe that equations \eqref{eq:diff distrib} and \eqref{eq:fL1 loc distrib} combined provide a flexible definition of the derivatives of any function $f \in L^1_{loc}(\calD)$ up to any order.
One also sees that \textit{weak} derivatives, as encountered in Sobolev spaces (\cite{brezis2010functional}, Section 9.1) and weak formulations of PDEs, are particular cases of distributional derivatives: given $\alpha\in\mathbb{N}^d$ and two locally integrable functions $f$ and $f_{\alpha}$, $f$ admits $f_{\alpha}$ for its $\alpha^{th}$ weak derivative if and only if $\partial^{\alpha}T_f = T_{f_{\alpha}}$. One then conveniently  writes $\partial^{\alpha}f = f_{\alpha}$ ($\partial^{\alpha}f$ is unique in $L_{loc}^1(\calD)$ from the injectivity of the mapping \eqref{eq:fL1 loc distrib}).

\subsubsection{Radon measures}\label{subsub:radon}
This subsection and the ones that follow are only necessary for dealing with the wave equation in Section \ref{Section:GP_wave}.
In this article, we call positive Radon measure any positive measure over $\calD$ that is Borel regular (\cite{evans2018measure}, Definition 1.9) and that has finite mass over any compact subset of $\calD$. A Radon measure is a linear combination of positive Radon measures. In \cite{lang1993}, Chapter 9, it is proved that
the space of Radon measures over $\calD$ is isomorphic to the space of continuous linear forms over $C_c(\calD)$, the space of compactly supported continuous functions on $\calD$ endowed with its usual LF-space topology described e.g. in \cite{treves2006topological}, pp. 131-133.
The corresponding isomorphism is given by
\begin{align}\label{eq:mu forme lin}
\mu \longmapsto T_{\mu} :
\begin{cases}
C_c(\calD) &\longrightarrow \mathbb{R}\\
\ \ f &\longmapsto \int_{\calD}f(x) \mu(dx). 
\end{cases}
\end{align}
We have the following facts. $(i)$ Any signed measure that admits a density $f$ with reference to the Lebesgue measure such that $f \in L^1_{loc}(\calD)$ is a Radon measure (\cite{treves2006topological}, p. 217). $(ii)$ Any Radon measure can be injectively identified to a generalized function $T_{\mu}$ by replacing $C_c(\calD)$ by $\scrD(\calD)$ in equation \eqref{eq:mu forme lin}. In particular, Radon measures can be differentiated up to any order through equation \eqref{eq:diff distrib}. $(iii)$ Any Radon measure $\mu$, can be uniquely written as $\mu = \mu^+ - \mu^-$ where $\mu^+$ and $\mu^-$ are positive Radon measures (\cite{lang1993}, Chapter 9). We then define its total variation by $|\mu| := \mu^+ + \mu^- $.

\subsubsection{Finite order generalized functions}\label{subsub:finite_order}
A generalized function $T$ is said to be of finite order if there exists a nonnegative integer $m$ such that one can take $n_K = m$, independently of $K$, in the definition of the continuity of $T$, i.e. equation \eqref{eq:def distrib continue}. The order of $T$ is then the smallest of those integers $m$. The space of generalized functions of order $m$ is isomorphic to $C_c^{m}(\calD)'$, the space of continuous linear forms over $C_c^{m}(\calD)$, when $C_c^{m}(\calD)$ is endowed with its usual LF-space topology (\cite{treves2006topological}, pp. 131-133).
The key property for us is that such generalized functions can be represented thanks to Radon measures. If $L$ is of order $m$, there exists a family of Radon measures $\{\mu_p\}_{|p|\leq m}$ over $\calD$ such that
\begin{align}\label{eq:finite order radon}
T = \sum_{|p|\leq m}\partial^p \mu_p,
\end{align}
where the equality in equation \eqref{eq:finite order radon} holds in $\scrD'(\calD)$ and $C_c^{m}(\mathcal{D)}'$ (\cite{treves2006topological}, p 259). Among the finite order generalized functions are those that are \textit{compactly supported}, i.e. those for which the measures $\mu_p$ such that $T = \sum_{|p|\leq m}\partial^p \mu_p$ all have compact support. 

\subsubsection{Convolution with generalized functions}\label{subsub:conv} As above, we consider $C_c^{m}(\mathbb{R}^d)$ endowed with its LF-space topology. Let $f \in C_c^{m}(\mathbb{R}^d)$ and $T \in C_c^{m}(\mathbb{R}^d)'$. Note $\tau_x f$ the function $y \longmapsto f(y-x)$ and $\check{f}$ the function $y \longmapsto f(-y)$. Then (\cite{treves2006topological}, p. 287, Section 27) one may define the convolution between $T$ and $f$ by
\begin{align}\label{eq:conv distrib}
T * f : x \longmapsto \langle T, \tau_{-x} \check{f} \rangle,
\end{align}
and $T * f$ is a function in the classical sense, i.e. defined pointwise. When $T$ is a regular generalized function, equation \eqref{eq:conv distrib} reduces to the usual convolution of functions through the identification defined in equation \eqref{eq:fL1 loc distrib}.
Similarly if $T$ is in fact a Radon measure $\mu$:
\begin{align}\label{eq:conv_mesure_func}
(T * f)(x) = \int_{\mathbb{R}^d} f(x-y)\mu(dy).
\end{align} 
More general definitions of generalized function convolution are available (\cite{treves2006topological}, Chapter 27) but this one is sufficient for our use.

\subsubsection{Tensor product of generalized functions}\label{subsub:tensor_gen_func} For two generalized functions $T_1 \in \scrD'(\calD_1)$ and $T_2 \in \scrD'(\calD_2)$, $T_1 \otimes T_2 \in \scrD'(\calD_1 \times \calD_2)$ denotes their tensor product(\cite{treves2006topological}, pp. 416-417), which is uniquely determined by the following tensor property:
\begin{align}
\forall \varphi_1 \in \scrD(\calD_1), \forall \varphi_2 \in \scrD(\calD_2), \ \langle T_1 \otimes T_2, \varphi_1 \otimes \varphi_2 \rangle = \langle T_1, \varphi_1 \rangle \times \langle T_2, \varphi_2 \rangle.
\end{align}
$T_1 \otimes T_2$ reduces to the tensor product of functions (respectively, measures) when $T_1$ and $T_2$ are functions (respectively, measures) through the identification of equation \eqref{eq:fL1 loc distrib} (respectively, equation \eqref{eq:mu forme lin}).

\section{Random fields under linear differential constraints}\label{Section:sto diff}
The results in this section state that under suitable assumptions over the first two moments of a given second order random field $U = (U(x))_{x\in \calD}$, sample path degeneracy properties with reference to differential constraints can be read on the first two moments of $U$, namely the mean function and the functions $k_x:y\longmapsto k(x,y)$, where $k$ is the covariance function of $U$. This is remarkable because the space induced by the sample paths of $U$ is a priori much larger than the space spanned by the functions $k_x, x\in \calD$. Moreover, the functions $k_x$ are ``accessible'', i.e. checking that these functions indeed verify the linear constraint can usually be done with direct computations.

We begin by recalling a result from \cite{ginsbourger2016} in the case of pointwise defined derivatives. We next state and prove a result similar to that of \cite{ginsbourger2016}, where we interpret the derivatives in the distributional sense.

\subsection{The case of classical derivatives}
We start by properly defining the notion of strong solutions of a PDE.
\begin{definition}[Strong/classical solutions]\label{def:strong_sol}
Let $L$ be a differential operator defined as in equation \eqref{eq:edp}, with continuous coefficients. We say that a function $u$ is a classical or strong solution to the PDE $L(u)=0$ if $u$ is $n$ times differentiable and $u$ verifies the PDE pointwise:
\begin{align}
\forall x \in \calD, \ \ \ L(u)(x) = \sum_{|\alpha| \leq n} a_{\alpha}(x) \partial^{\alpha}u(x) = 0.
\end{align}
\end{definition}
Note that the space of $n$ times differentiable functions does not have the nice topological properties of $C^n(\calD)$ and in most cases met in practice, one rather requires that strong solutions lie in $C^n(\calD)$. It is however in the sense of the definition \ref{def:strong_sol} that the theorem from \cite{ginsbourger2016} is best understood. This theorem, which we remind in Proposition \ref{prop : lin constraints classic}, is the one proved and used in \cite{ginsbourger2016} to build a Gaussian process whose sample paths are all strong solutions to the Laplace equation on a 2D circular domain. 

We first introduce some notations. Let $(U(x))_{x\in \calD}$ be a centered Gaussian process with covariance function $k$. Denote $\calF(\calD,\mathbb{R})$ the space of real-valued pointwise-defined functions on $\calD$ (often alternatively denoted $\mathbb{R}^{\mathcal{D}}$). We will only use $\calF(\calD,\mathbb{R})$ as a set, therefore we do not consider any topology over it. We refer to \cite{schwartz1964sous}, Section 9, for details on $\calF(\calD,\mathbb{R})$ seen as a topological vector space.
Denote $\calH_k$ the reproducing kernel Hilbert space (RKHS) associated to $k$ (see \cite{agnan2004}, Definition 1 p. 7 and Theorem 3, p. 19). $\calH_k$ is a Hilbert space of pointwise-defined functions (i.e. $\calH_k \subset \calF(\calD,\mathbb{R})$ as sets), such that the pointwise evaluation maps $l_x :f \mapsto f(x)$ are continuous functionals. Although belonging to $\calF(\calD,\mathbb{R})$ does not seem very restrictive at first glance, this clashes with the usual $L^p$ and Sobolev spaces encountered in PDE theory, which are sets of functions defined up to a set of null Lebesgue measure.

\begin{proposition}[sample paths of GPs under linear constraints \cite{ginsbourger2016}]\label{prop : lin constraints classic}
Let $\big(U(x)\big)_{x\in \calD}  \sim GP(0,k)$. Note for all $x \in \calD$ the function $k_x : y \longmapsto k(x,y)$. Let $E$ be a real vector space of functions defined on $\calD$ that contains the sample paths of $U$ almost surely and $L : E \longrightarrow \calF(\calD,\mathbb{R})$ be a linear operator. Assume that for all $x \in \calD,\ L(U)(x) \in \mathcal{L}(U)$, where $\mathcal{L}(U)$ is the closure of  $\text{Span}\{U(x), x \in \calD\}$ in $L^2(\mathbb{P})$. Then there exists a unique linear operator $\scrL : \calH_k \longrightarrow \calF(\calD,\mathbb{R})$ such that
\begin{align*}
    \forall x,y \in \calD,\ \ \ \mathbb{E}[L(U)(x)U(y)] = \scrL(k_{y})(x),
\end{align*}
and such that for all $x \in \calD, h \in \calH_k$ and sequence $(h_n) \subset \calH_k$ such that $ h_n \xrightarrow[]{} h$ for the topology of $\calH_k$, we have $ \ \scrL(h_n)(x) \longrightarrow \scrL(h)(x)$. Finally, the following statements are equivalent:
\begin{enumerate}[label=(\roman*)]
    \item $\mathbb{P}(L(U)= 0) = 1$.
    \item $\forall x \in \calD, \scrL(k_{x}) = 0$.
\end{enumerate}
\end{proposition}
A sufficient condition ensuring that the sample paths of a GP lie in $C^n(\calD)$ is found in \cite{adler2007}, Theorem 1.4.2. More broadly, both necessary and sufficient conditions over the first two moments of a GP for its sample path  to be (Hölder) continuous are well-known: see e.g. \cite{adler1981}, Theorems 3.3.3 and 8.3.2. 

The proof of Proposition \ref{prop : lin constraints classic} heavily relies on the Loève isometry (\cite{agnan2004}, Theorem 35, p. 65) between the two Hilbert spaces $\calH_k$ and $\mathcal{L}(U)$ (see Section \ref{subsub:second_GP} for details on $\mathcal{L}(U)$). This theorem can be applied when $L$ is a differential operator as discussed in \cite{ginsbourger2016}. However, in Proposition \ref{prop : lin constraints classic}, the differential operator $L$ of order $n$ has to be valued in $\calF(\calD,\mathbb{R})$; in particular for $u \in E$, the function $L(u)$ has to be defined pointwise in order to use the Loève isometry. To summarize, in all generality the derivatives in $L$ have to be understood in a classical sense and $E$ has to be contained in $\scrD^n(\calD)$, the space of $n$ times differentiable functions on $\calD$. Requiring that $E \subset \scrD^n(\calD)$ is a very strong assumption with reference to the sample paths of $U$; furthermore, this is not compliant with the usual way of studying PDEs where derivatives are understood in a weaker sense. We present in Proposition \ref{prop :diff_constraints_distrib} an adaptation of Proposition \ref{prop : lin constraints classic} where the derivatives are understood in the distributional sense. By transferring all the derivatives on the test function, we will be liberated from any differentiability assumptions over the sample paths of $U$, effectively replacing $\scrD^n(\calD)$ with $L_{loc}^1(\calD)$. Finally, the random field $U$ will not be assumed Gaussian and will only be required to be measurable second order.

\subsection{The case of distributional derivatives}
\subsubsection{Distributional solutions of PDEs}
In this section, we elaborate a bit more on the notion of distributional solutions to a given PDE. Let $L = \sum_{|\alpha| \leq n} a_{\alpha}(x) \partial^{\alpha}$ be a linear differential operator, and assume for the moment that its coefficients are infinitely differentiable. We briefly recall the steps described in the introduction that lead to the definition of distributional solutions presented in equation \eqref{eq:def_sol_distribs}. Start from a strong solution $u$  of class $C^n$ of $L(u)=0$, multiply this PDE by a test function $\varphi \in \scrD(\calD)$, integrate over $\calD$ and perform $|\alpha|$ integration by parts to transfer all derivatives from $u$ to $\varphi$. Since the support of $\varphi$ is a compact subset of the open set $\calD$, the boundary terms of each integration by parts vanish, leading to
\begin{align}\label{eq:def_sol_distrib_2}
\forall \varphi \in \scrD(\calD), \ \ \ \int_{\calD}u(x)\sum_{|\alpha| \leq n}(-1)^{|\alpha|} \partial^{\alpha}(a_{\alpha}\varphi)(x)dx = 0.
\end{align}
Following equation \eqref{eq:def_sol_distrib_2} we introduce $L^*$, the formal adjoint of $L$, acting on $\scrD(\calD)$, defined by the following formula (\cite{treves2006topological}, pp. 247-249)
\begin{align}\label{eq:def_Lstar}
L^* : \varphi \longmapsto \sum_{|\alpha| \leq n}(-1)^{|\alpha|} \partial^{\alpha}(a_{\alpha}\varphi).
\end{align}
Note that for equation \eqref{eq:def_sol_distrib_2} to be well defined, the assumptions that $u \in L_{loc}^1(\calD)$ and $a_{\alpha} \in C^{|\alpha|}(\calD)$ are sufficient. More precisely, these assumptions are enough to show that the map $L(u)$ defined by duality
\begin{align}\label{eq:def_TU_distrib}
L(u): \begin{cases}
\scrD(\calD) &\longrightarrow \mathbb{R} \\
\varphi &\longmapsto \int_{\calD}L^*(\varphi)(x)u(x)dx
\end{cases}
\end{align}
defines a continuous linear form over $\scrD(\calD)$, i.e. $L(u) \in \scrD'(\calD)$ (see equations \eqref{eq:def distrib continue} and \eqref{eq:control inf} for a rigorous proof of this statement). This definition extends the definition of distributional derivatives from Section \ref{subsub:diff_of_distrib} to differential operators.
By construction, $L$ and $L^*$ verify a duality identity: given $\varphi \in \scrD(\calD)$ and $u\in L_{loc}^1(\calD), \ \langle L(u),\varphi\rangle = \langle u,L^*(\varphi)\rangle$.

As in Section \ref{subsub:regular_distrib}, the assumption that $u \in L_{loc}^1(\calD)$ is in fact a \textit{continuity} assumption over the associated linear form $L(u)$ (a more general and theoretical analysis of such observations can be found in \cite{treves2006topological}, pp. 247-251). This finally leads to the following definition, following e.g.  \cite{Duistermaat2010}, p. 10:
\begin{definition}[Distributional solutions]\label{def:distrib_sol}
A function $u \in L_{loc}^1(\calD)$ is said to be a solution to the PDE $L(u) = 0$ in the sense of distributions if $L(u) = 0$ in $\scrD'(\calD)$, i.e. when $L(u)$ is seen as en element of $\scrD'(\calD)$ through equation \eqref{eq:def_TU_distrib} and $0$ is the null linear form over $\scrD(\calD)$.
\end{definition}
As weak derivatives are a particular case of distributional derivatives (Section \ref{subsub:regular_distrib}), one expects that the distributional solutions of a PDE that admit some weak derivatives are in fact weak solutions, i.e. solutions of some weak formulation of that PDE. Rigorous statements of this general fact have to be checked on a case-by-case basis, depending on the weak formulation at hand (a more in-depth discussion falls outside of the scope of this article). As an example, this is the case for the weak formulation of elliptic PDEs in $H_0^1(\calD)$ (see e.g. \cite{evans1998}, Section 6.2), where $H_0^1(\calD)$ is the closure of $\scrD(\calD)$ in the Sobolev space $H^1(\calD) := \{u\in L^2(\calD): \nabla u \text{ exists as a weak derivative and } \nabla u \in L^2(\calD)^d\}$.

\begin{remark}[Measure-valued solutions of PDEs]\label{rk:meas_val_sol}
Although it is not the main focus of the paper, we can even allow $u$ in Definition \ref{def:distrib_sol} to be a Radon measure by replacing $u(x)dx$ with $\mu(dx)$ in equation \eqref{eq:def_TU_distrib}. This will be useful in Section \ref{sub:sol_3D_wave}, where we will encounter a  measure-valued PDE solution which is central from a physical viewpoint, with the wave equation's Green's function (it is \textit{not} actually a function!). Notice that weak formulations in Sobolev spaces, say $H^1(\calD)$, are not well-equipped to work with such solutions, and our distributional framework becomes needed. \end{remark}

\subsubsection{Random fields under distributional differential constraints} We can now state the following proposition, based on Definition \ref{def:distrib_sol}.
\begin{proposition}[sample paths of random fields under linear differential constraints, distributional derivatives] \label{prop :diff_constraints_distrib} Let $\calD \subset \mathbb{R}^d$ be an open set and let $L = \sum a_{\alpha}(x) \partial^{\alpha}, \ {|\alpha| \leq n},$ be a linear differential operator of order $n$ with coefficients $a_{\alpha}(x) \in C^{|\alpha|}(\calD)$. Let $U = \big(U(x)\big)_{x\in \calD}$ be a measurable second order random field with mean function $m(x)$ and covariance function $k(x,x')$. For all $x \in \calD$, note ${k_{x} : y \longmapsto k(x,y)}$. Suppose that $m \in L^1_{loc}(\calD)$ and $\sigma\in L_{loc}^1(\calD)$, where ${\sigma}: x \mapsto k(x,x)^{1/2}$. \\
1) Then $\mathbb{P}(U \in L_{loc}^1(\calD)) = 1$ and for all $x \in \calD, k_x \in L_{loc}^1(\calD)$. \\
2) Suppose that $L(m) = 0$ in the sense of distributions. Then the following statements are equivalent:
\begin{enumerate}[label=(\roman*)]
    \item $\mathbb{P}(L(U) = 0 \text{ in the sense of distributions}) = 1$.
    \item $\forall x \in \calD,\ L(k_{x}) = 0$ in the sense of distributions.
\end{enumerate}
\end{proposition}
Explicitly, by $(i)$ we mean that there exists a set $A \in \calA$ with $\mathbb{P}(A) = 1$ such that for all $\omega \in A$,
\begin{align}\label{eq:meaning i}
 \forall \varphi \in \scrD(\calD),\ \ \langle U_{\omega}, L^*\varphi \rangle = \int_{\calD}U_{\omega}(x)L^*\varphi(x)dx = 0  .
\end{align}
The fact that the functions $x \longmapsto U_{\omega}(x)$ and $y \longmapsto k_x(y)$ lie in $L^1_{loc}(\calD)$ ensure the existence of the integrals in equations \eqref{eq:meaning i} (see Point 2 of the proof of Proposition \ref{prop :diff_constraints_distrib}) as well as the continuity of the associated linear forms over $\scrD(\calD)$, following the definition of equation \eqref{eq:def_TU_distrib}. The assumption that $a_{\alpha}\in C^{|\alpha|}(\calD)$ is not very strong, in the sense that it is the minimal assumption to ensure that the adjoint $L^*$ is well-defined (equation \eqref{eq:def_Lstar}), and thus that Definition \ref{def:distrib_sol} even makes sense. Likewise, requiring that $\sigma \in L_{loc}^1(\calD)$ is not very restrictive (see Section \ref{subsub:L1loc_test}). However, ensuring the measurability of the random process $U$ is more demanding in practice, because it is difficult to ensure this property outside of having continuity in probability (see Section \ref{subsub:sto_pro_mes}).

The following lemma will be crucial for the proof of Proposition \ref{prop :diff_constraints_distrib}:
\begin{lemma}\label{lemma:d_separable}
$\scrD(\calD)$ is sequentially separable, i.e. there exists a countable subset $F \subset \scrD(\calD)$ such that for all $\varphi\in\scrD(\calD)$, there exists a sequence $(\varphi_n)\subset F$ such that $\varphi_n\rightarrow\varphi$ in $\scrD(\calD)$ for its LF topology.
\end{lemma}
Recall that a topological space $E$ is separable if there exists a countable subset $F\subset E$ such that its closure in $E$ is equal to $E$. If the topology of $E$ is metrizable (as e.g. for Fréchet spaces), sequential separability and separability are equivalent. If this topology is \textit{not} metrizable (as e.g. for LF spaces), then sequential separability implies separability but the converse need not hold. Below, we provide a short proof of Lemma \ref{lemma:d_separable}, as we could not find it in the literature. The weaker property that $\scrD(\calD)$ is separable is already difficult to track down, see e.g. \cite{gapaillard8processus}, Corollaire (1).2, p. 78 or \cite{gelfand1964generalized}, p. 73, (3).
\begin{proof}
We first show that the spaces $\scrD_{K_i}(\calD)$ introduced in Section \ref{subsub:gen_func} are separable Fréchet spaces. The Fréchet topology of $\scrD_{K_i}(\calD)$ is the one induced by the usual Fréchet topology of $C^{\infty}(\calD)$ when $\scrD_{K_i}(\calD)$ is seen as a subspace of $C^{\infty}(\calD)$ (\cite{treves2006topological}, pp. 131-132). 
As a Fréchet space, $C^{\infty}(\calD)$ is metrizable (\cite{treves2006topological}, p. 85). But $C^{\infty}(\calD)$ is also a Montel space (\cite{treves2006topological}, Proposition 34.4, p. 357): as a metrizable Montel space, it is automatically separable (\cite{schaefer1999}, p. 195 or \cite{dieudonne_separable}). Thus $\scrD_{K_i}(\calD)$ is also separable as a subset of the separable metrizable space $C^{\infty}(\calD)$ (\cite{brezis2010functional}, Proposition 3.25, p. 73).

Denote now $F_i$ a countable dense subset of $\scrD_{K_i}(\calD)$ and consider $F:= \bigcup_{i\in\mathbb{N}}F_i$. Let $\varphi\in \scrD(\calD)$ and $i\in\mathbb{N}$ such that $\text{Supp}(\varphi) \subset K_i$, where $(K_i)_{i\in\mathbb{N}}$ is the sequence of compact sets from Section \ref{subsub:gen_func}. Then $\varphi \in \scrD_{K_i}(\calD)$ and there exists a sequence $(\varphi_n) \subset F_i \subset F$ such that $\varphi_n \rightarrow \varphi$ in the sense of the Fréchet topology of $\scrD_{K_i}(\calD)$, i.e. the metric $d_i$ in equation \eqref{eq:def_di}. From equation \eqref{eq:def_di}, $||\partial^{\alpha}\varphi_n-\partial^{\alpha}\varphi||_{\infty} \rightarrow 0$ for all $\alpha\in \mathbb{N}^d$. Since $\text{Supp}(\varphi_n) \subset K_i$ for all $n\in\mathbb{N}$, we have that $\varphi_n\rightarrow\varphi$ in $\scrD(\calD)$ (see Section \ref{subsub:gen_func}).
\end{proof}
We are now able to prove Proposition \ref{prop :diff_constraints_distrib}.
\begin{proof}
Suppose first that $U$ is centered, i.e. $m \equiv 0$. \\
1) We begin by showing that the sample paths of $U$ almost surely lie in $L^1_{loc}(\calD)$. Note first that thanks to the Cauchy-Schwarz inequality, $\mathbb{E}[|U(x)|] \leq \sigma(x)$.
Now, let $(K_i)_{i \in \mathbb{N}}$ be an increasing sequence of compact subsets of $\calD$ such that $\bigcup_{i \in \mathbb{N}} K_i = \calD$. Using Tonelli's theorem, we have that for any $n \in \mathbb{N}$,
\begin{align}\label{eq:int_finie_kn}
\mathbb{E}\bigg[\int_{K_i}|U(x)|dx\bigg] = \int_{K_i}\mathbb{E}[|U(x)|]dx \leq \int_{K_i} \sigma(x)dx < +\infty,
\end{align}
since $\sigma \in L^1_{loc}(\calD)$. Note that in order for the integrals above to be well defined, imposing that $U$ is a measurable random field cannot be circumvented. Equation \eqref{eq:int_finie_kn} yields a set $B_n \subset \Omega$ of probability 1 over which the random variable
$\omega \longmapsto \int_{K_i}|U_{\omega}(x)|dx$ is finite (from Fubini's theorem again, the map $\omega \longmapsto \int_{K_i}|U_{\omega}(x)|dx$ is measurable). Consider now the set $B = \bigcap_{n \in \mathbb{N}}B_n$ which remains of probability 1. For all compact subset $K \subset \calD$, there exists an integer $n_K$ such that $K \subset K_{n_K}$ and thus for all $\omega \in B$,
\begin{align}\label{eq:u_w_L1loc}
\int_K |U_{\omega}(x)|dx \leq \int_{K_{n_K}} |U_{\omega}(x)|dx < + \infty,
\end{align}
which shows that the sample paths of $U$ lie in $L^1_{loc}(\calD)$ almost surely. Similarly, we check that for all $x \in \calD,  k_{x}$ lies in $L^1_{loc}(\calD)$: for any compact set $K$, since $\sigma \in  L^1_{loc}(\calD)$ and because of equation \eqref{eq:CS PD},

\begin{align*}
\int_K |k_{x}(y)|dy = \int_K |k(x,y)|dy \leq \sigma(x) \int_K \sigma(y)dy < \infty.
\end{align*}

2) Let us check in advance that whatever $f \in L^1_{loc}(\calD)$, the map $T(f): \varphi \longmapsto \langle f, L^*\varphi \rangle$ is a continuous linear form over $\scrD(\calD)$. Since $a_{\alpha} \in C^{|\alpha|}(\calD)$, we can apply Leibniz' rule on $L^*\varphi = \sum_{|\alpha|\leq n} (-1)^{|\alpha|}\partial^{\alpha}(a_{\alpha}\varphi)$. This yields a family $\{f_{\alpha}\}_{|\alpha| \leq n}$ of continuous functions over $\calD$ such that
\begin{align}\label{eq:les f_alpha}
\forall \varphi \in \scrD(\calD), \ \ \forall x \in \calD,\ \ L^*\varphi(x) = \sum_{|\alpha|\leq n}f_{\alpha}(x) \partial^{\alpha}\varphi(x).
\end{align}
For all $f \in L^1_{loc}(\calD)$, for all compact set $K \subset \calD$ and for all $\varphi \in \scrD(\calD)$ such that $\text{Supp}(\varphi) \subset K$, we have $\text{Supp}(L^*\varphi) \subset K$ and equation \eqref{eq:les f_alpha} yields
\begin{align}\label{eq:control inf} 
|\langle f, L^*\varphi \rangle| &\leq \int_{\calD}|f(x)||L^*\varphi(x)|dx \nonumber \\ 
&\leq \bigg( \int_{K}|f(x)|dx \times \max_{|\alpha|\leq n} \sup_{x \in K} |f_{\alpha}(x)|\bigg) \times \sum_{|\alpha| \leq n} ||\partial^{\alpha}\varphi ||_{\infty} < +\infty.
\end{align}
This proves that $T(f) : \varphi \longmapsto \langle f, L^*\varphi \rangle$ is a continuous linear form over $\scrD(\calD)$ (see equation \eqref{eq:def distrib continue}).

\underline{$(i) \implies (ii)$}: Suppose $(i)$. Let $\varphi \in \scrD(\calD)$. There exists a  set $A\subset \Omega$ such that $\mathbb{P}(A) = 1$ and such that
\begin{align*}
    \forall \omega \in A, \ \ \ \langle U_{\omega},L^*\varphi\rangle = \int_{\calD}U_{\omega}(x)L^*\phi(x)dx = 0.
\end{align*} 
Multiplying equation above with the random variable $U(x')$, taking the expectation and formally permuting (for now) the integral and the expectation, we obtain
\begin{align*}
   0 &= \mathbb{E}\bigg[U(x')\int_{\calD}U(x)L^*\varphi(x)dx\bigg] = \int_{\calD}L^*\varphi(x)\mathbb{E}[U(x)U(x')]dx \\
     &= \int_{\calD}L^*\varphi(x)k(x,x')dx = \langle k_{x'},L^*\varphi \rangle.
\end{align*}
The integral-expectation permutation is justified by writing down the expectation as an integral and using Fubini's theorem, checking that the below quantity is finite. We use Tonelli's theorem and the Cauchy-Schwarz inequality:
\begin{align*}
    \mathbb{E}\bigg[\int_{\calD}|U(x')U(x)L^*\varphi(x)|dx\bigg] &= \int_{\calD}|L^*\varphi(x)|\mathbb{E}[|U(x)U(x')|]dx\\ 
    &\leq \int_{\calD}|L^*\varphi(x)|\mathbb{E}[U(x)^2]^{1/2}\mathbb{E}[U(x')^2]^{1/2}dx \\
    &\leq \sigma(x')\int_{\calD}|L^*\varphi(x)|\sigma(x) dx < +\infty.
\end{align*}
Indeed, since $\sigma \in L^1_{loc}(\calD)$, setting $f = \sigma$ in equation \eqref{eq:control inf} shows that the last integral is indeed finite.
Thus, $\forall x \in \calD, \forall \varphi \in \scrD(\calD), \langle k_{x}, L^*\varphi \rangle = 0$ which proves that $(i) \implies (ii)$. \\
\underline{$(ii) \implies (i)$}: Suppose $(ii)$. Let $\varphi \in \scrD(\calD)$, we have $\langle k_{x'}, L^*\varphi \rangle = 0$. Multiplying this with $L^*\varphi(x')$ and integrating with reference to $x'$ yields
\begin{align*}
    0 = \int_{\calD}L^*\varphi(x') \int_{\calD}L^*\varphi(x)k(x,x')dxdx' =  \int_{\calD} \int_{\calD}L^*\varphi(x)L^*\varphi(x')\mathbb{E}[U(x)U(x')]dxdx'.
\end{align*}
Permuting formally the expectation and the integrals (justified in equation \eqref{eq:justify fubini prop2}) yields
\begin{align*}
   0 &=  \int_{\calD} \int_{\calD}L^*\varphi(x)L^*\varphi(x')\mathbb{E}[U(x)U(x')]dxdx' \\
    &= \mathbb{E}\Bigg[\Bigg(\int_{\calD}L^*\varphi(x)U(x)dx\Big)^2\Bigg] = \mathbb{E}\big[\langle U,L^*\varphi\rangle ^2\big],
\end{align*}
and thus $\langle U,L^*\varphi\rangle = 0 $ a.s. : there exists $A_{\varphi} \in \calA$ with $\mathbb{P}(A_{\varphi}) = 1$ such that $\forall \omega \in A_{\varphi}, \langle U_{\omega}, L^*\varphi \rangle = 0$. We justify the expectation-integral permutation with the computation below
\begin{align} \label{eq:justify fubini prop2}
    \int_{\calD}\int_{\calD}|L^*\varphi(x)&L^*\varphi(x')|\mathbb{E}[|U(x)U(x')|]dxdx' \nonumber\\
    &\leq \int_{\calD}\int_{\calD}|L^*\varphi(x)L^*\varphi(x')|\sigma(x)\sigma(x')dxdx' \nonumber \\
     &\leq \Bigg(\int_{\calD}|L^*\varphi(x)|\sigma(x)dx \Bigg)^2 < + \infty.
\end{align}
As previously, setting $f = \sigma$ in equation \eqref{eq:control inf} shows that the integral above is indeed finite.

This does not finish the proof as we need to find a set $A$ with $\mathbb{P}(A) = 1$, independently from $\varphi$, such that $\forall \omega \in A, \langle U_{\omega}, L^*\varphi \rangle = 0$.
For this we shall use Lemma \ref{lemma:d_separable}. Let
\begin{align}\label{eq:A_prob_1}
A := B \cap\big( \bigcap_{\varphi \in F}A_{\varphi}\big),
\end{align}
where the set $F$ is introduced in Lemma \ref{lemma:d_separable}.
Then $\mathbb{P}(A)=1$ since $\mathbb{P}(B) = 1, \mathbb{P}(A_{\varphi}) = 1$ and $F$ is countable. Let $\omega \in A$. Since $U_{\omega} \in L^1_{loc}(\calD)$, equation \eqref{eq:control inf} shows that the map $T_{\omega} : \varphi \longmapsto \langle U_{\omega},L^*\varphi\rangle$ is a continuous linear form on $\scrD(\calD)$. In particular, Theorem 6.6$(c)$ p. 155 from \cite{rudin1991} states that $T_{\omega}$ is in fact \textit{sequentially} continuous. Let $\varphi \in \scrD(\calD)$ and $(\varphi_n)\subset F$ be such that $\varphi_n \rightarrow \varphi$ in $\scrD(\calD)$, from Lemma \ref{lemma:d_separable}. From the sequential continuity of $T_{\omega}$, $T_{\omega}(\varphi) = \text{lim}_{n\rightarrow \infty}  T_{\omega}(\varphi_n) = 0$ since $\forall n\in\mathbb{N}, T_{\omega}(\varphi_n) = 0$. That is, we have proved that
\begin{align*}
    \forall \omega \in A,\ \ \forall \varphi \in \scrD(\calD),\ \ \langle U_{\omega},\
    L^*\varphi \rangle = T_{\omega}(\varphi) = 0.
\end{align*}
Since $\mathbb{P}(A) = 1$, this shows that $(ii) \implies (i)$. 

When $U$ is not centered, consider the centered random field $V$ defined by $V(x) = U(x) - m(x)$ for which the above proof can be applied. Since $L$ is linear and $m$ is assumed to verify $L(m) = 0$ in the sense of distributions, the probabilistic sets $A_U = \{L(U) = 0 \text{ in the sense of distributions}\}$ and $A_V = \{L(V) = 0 \text{ in the sense of distributions}\}$ coincide and thus, $A \subset A_U$.
Finally, $U$ and $V$ have the same covariance function $k(x,x')$. Thus, 
\begin{align*}
\mathbb{P}(L(U) = 0 \text{ in the distrib. sense}) = 1 &\iff \mathbb{P}(L(V) = 0 \text{ in the distrib. sense}) = 1 \\
&\iff \forall x \in \calD, L(k_{x}) = 0 \text{ in the distrib. sense},
\end{align*}
which finishes the proof in the general case.
\end{proof}

\begin{remark}
Distributional solutions are the weakest types of solutions for PDEs. In general, additional regularity conditions have to be imposed to obtain physically realistic solutions, such as Sobolev regularity or entropy conditions as for nonlinear hyperbolic PDEs \cite{Serre1999SystemsOC}. However, every step in the above proof remains valid when replacing $\varphi \in \scrD(\calD)$ with $\varphi \in C_c^n(\calD)$. Although we have not explicited the usual topology of $C_c^n(\calD)$ in this article, we state that this is enough to show that the equalities stated in Proposition \ref{prop :diff_constraints_distrib} also hold in $C_c^n(\calD)'$, the space of finite order generalized functions of order $n$, rather than just in $\scrD'(\calD)$. $C_c^n(\calD)'$ is a smaller space than $\scrD'(\calD)$, though less used in functional analysis than $\scrD'(\calD)$.
\end{remark}
We partially recover Proposition \ref{prop : lin constraints classic} when the sample paths of $U$ are $n$ times differentiable with locally integrable $n^{th}$ derivative and $k \in C^{n,n}(\calD\times \calD)$. Indeed, in that case one can show that if $L = \sum_{|\alpha|\leq n} a_{\alpha}(x)\partial^{\alpha}$, then we simply have $\scrL = L$ in Proposition \ref{prop : lin constraints classic}. Additionally, $L(U_{\omega})$ and $L(k_x)$ both lie in $\calF(\calD,\mathbb{R}) \cap L^1_{loc}(\calD)$. In that framework, Proposition \ref{prop : lin constraints classic} states that
\begin{align}\label{eq:res prop 1}
\forall x \in \calD, \ L(k_x) = 0 \iff \mathbb{P}(L(U) = 0) = 1,
\end{align}
where the function equalities of the form $L(f) = 0$ in equation \eqref{eq:res prop 1} are valid everywhere on $\calD$. In contrast, for any function $g$ that lies in $L^1_{loc}(\calD)$, we have
\begin{align}\label{eq:g 0 ae}
g = 0 \text{\ in the sense of distributions} \iff
g = 0 \ a.e.
\end{align}
Equation \eqref{eq:g 0 ae} is just another way of saying that the linear map $f \longmapsto T_f$ given in \eqref{eq:fL1 loc distrib} is injective. Following equation \eqref{eq:g 0 ae}, Proposition \ref{prop :diff_constraints_distrib} states a slightly weaker result than \eqref{eq:res prop 1}, namely that
\begin{align}\label{eq:recover_ae}
\forall x \in \calD, \ \ L(k_x) = 0 \ a.e. \iff \mathbb{P}(L(U) = 0 \ a.e.) = 1.
\end{align}
If we actually have that the sample paths of $U$ lie in $C^n(\calD)$, nullity almost everywhere implies nullity everywhere and we recover equation \eqref{eq:res prop 1} from equation \eqref{eq:recover_ae}.

Instead of having the sample paths of $U$ lie in $C^n(\calD)$ though, one may rather encounter the case where $U$ is \textit{mean-square} differentiable up to a certain order $m$. Under some continuity assumptions over the covariance function of $U$ and up to suitable modifications, \cite{SCHEUERER2010} showed that the sample paths of the mean-square differentiated process are actually weak derivatives of the sample paths of $U$. As observed after Definition \ref{def:distrib_sol}, we thus expect that the sample paths of the mean-square differentiable random fields verifying Point $2,(ii)$ of Proposition \ref{prop :diff_constraints_distrib} are solutions of some weak formulation of the PDE, rather than just distributional solutions.
\begin{example}[A first order PDE]
Consider a continuous, nondifferentiable one dimensional covariance function $k_0 : \mathbb{R}\times\mathbb{R} \rightarrow \mathbb{R}$, for example $k_0(x,x') = \exp(-|x-x'|)$. It is then readily checked that the function ${k : \mathbb{R}^2\times\mathbb{R}^2\rightarrow \mathbb{R}}$ defined by $k((x,y),(x',y')) = k_0(x-y,x'-y')$ is positive definite and verifies Point $2,(ii)$ of Proposition \ref{prop :diff_constraints_distrib} for the PDE
\begin{align}\label{eq:simple_pde}
    \partial_xu + \partial_yu = 0 \ \ \ \text{ in } \mathbb{R}^2.
\end{align}
Consider now a centered second order random field $(U(x,y))_{(x,y)\in\mathbb{R}^2}$ with covariance function $k$, passing to a measurable version of $U$ if necessary (it exists from Section \ref{subsub:sto_pro_mes}, as the continuity of $k$ yields the continuity in probability of $U$). Then almost surely, its sample paths verify the PDE \eqref{eq:simple_pde} in the sense of distributions, even though they are not expected to be differentiable. An example of random field whose covariance function is $k$ as defined above, is the GP $(U_0(x-y))_{(x,y)\in\mathbb{R}^2}$ where $(U_0(x))_{x\in\mathbb{R}} \sim GP(0,k_0)$. These formulas can be obtained by viewing the PDE \eqref{eq:simple_pde} as a transport equation under the condition that $U(x,0) = U_0(x)$, following the same approach as in the upcoming Section \ref{sub:gp_model_wave}.
\end{example}
\subsection{A heredity property for Gaussian process regression} 
\subsubsection{Gaussian process regression in a nutshell}\label{subsub:GPR}
GPs can be used for function interpolation. Let $u$ be a function defined on $\calD$ for which we know a dataset of values $B = \{u(x_1),...,u(x_n)\}$. Conditioning the law of a GP $(U(x))_{x \in \calD} \sim GP(m,k)$ on the database $B$ yields a second GP $\Tilde{U}$ given by
$ \Tilde{U}(x) := (U(x) | U({x_i}) = u(x_i), i = 1,...,n)$.
The law of $\Tilde{U}$ is known: $(\Tilde{U}(x))_{x \in \calD} \sim GP(\Tilde{m},\tilde{k})$. $\Tilde{m}$ and $\tilde{k}$ are given by the so-called \textit{Kriging} equations \eqref{eq:krig mean} and \eqref{eq:krig cov}. Let $X = (x_1,...,x_n)^T$, denote $m(X)$ the column vector such that $m(X)_i = m(x_i)$, $k(X,X)$ the square matrix such that $k(X,X)_{ij} = k(x_i,x_j)$ and given $x \in \calD$, $k(X,x)$ the column vector such that $k(X,x)_i = k(x_i,x)$. Suppose that $K(X,X)$ is invertible, then \cite{gpml2006}
\begin{numcases}{}
    \Tilde{m}(x) &= \hspace{3mm}$m(x) + k(X,x)^Tk(X,X)^{-1}(u(X) - m(X))$, \label{eq:krig mean} \\
    \tilde{k}(x,x') &= \hspace{3mm}$k(x,x') - k(X,x)^Tk(X,X)^{-1}k(X,x')$. \label{eq:krig cov}
\end{numcases}
The Kriging standard deviation function is then given by 
\begin{align}\label{eq:def_std_krig}
\tilde{\sigma}(x) = \tilde{k}(x,x)^{1/2}.
\end{align}
The so-called Kriging mean $\tilde{m}$ plays the role of an approximation of $u$; in particular, it interpolates $u$ at the observation points: $\tilde{m}(x_i)=u(x_i)$ for all $i = 1, ... , n$. Moreover, the Kriging covariance $\tilde{k}$ can be used to further control the distance between $u$ and $\tilde{m}$.
\subsubsection{Conditioned Gaussian processes under linear differential constraints}
We can now state the following corollary, which draws the consequences of Proposition \ref{prop :diff_constraints_distrib} when applied to GPR.

\begin{proposition}[Heredity of Proposition \ref{prop :diff_constraints_distrib} to conditioned GPs]\label{prop : inheritance pde conditioned}
Let $\calD$ and $L$ be as defined in Proposition \ref{prop :diff_constraints_distrib}. Let $(U(x))_{x\in \calD} \sim GP(m,k)$ be a Gaussian process that verifies the assumptions of Proposition \ref{prop :diff_constraints_distrib}.
Suppose also that
\begin{align}\label{eq:m k_x sol}
L(m) = 0 \text{ \ and \ } \forall x \in \calD,\ L(k_x) = 0 \text{ \ both in the sense of distributions}.
\end{align}
$(i)$ Then whatever the integer $p$, the vector $u = (u_1,...,u_p)^T \in \mathbb{R}^p$ and the vector $X = (x_1,...,x_p)^T \in \calD^p$ such that $k(X,X)$ is invertible, the Kriging mean $\tilde{m}(x)$ and the Kriging standard deviation function $\tilde{\sigma}$ both lie in $L^1_{loc}(\calD)$, and we have
\begin{align*}
L(\tilde{m}) = 0 \text{ \ and \ } \forall x \in \calD, \ L(\tilde{k}_x) = 0 \text{ \ both in the sense of distributions}.
\end{align*}
where $\tilde{m}$ and $\tilde{k}$ are defined in equations \eqref{eq:krig mean} and \eqref{eq:krig cov}. \\
$(ii)$ As such, the sample paths of the conditioned Gaussian process $\big(\tilde{U}(x)\big)_{x\in \calD}$ defined by $\tilde{U}(x) = (U(x)|U(x_i) = u_i \ \forall i = 1,...,p)$ are almost surely solutions of the equation $L(f) = 0$ in the sense of distributions:
\begin{align*}
\mathbb{P}(L(\tilde{U}) = 0 \text{ in the sense of distributions}) = 1.
\end{align*}
\end{proposition}

\begin{proof}
Note first that for all $x \in \calD, \tilde{k}(x,x) \leq k(x,x)$ (\cite{fasshauer2007}, p. 117). Thus the function $\tilde{\sigma} : x \longmapsto \tilde{k}(x,x)^{1/2}$ also lies in $L^1_{loc}(\calD)$.
Point $(i)$ is then a direct consequence of the definition of $\tilde{m}$ and $\tilde{k}$ in equations \eqref{eq:krig mean} and \eqref{eq:krig cov}, and the linearity of $L$. 
Proposition \ref{prop :diff_constraints_distrib} can then be applied conjointly with $(i)$, which yields point $(ii)$ since the mean and covariance functions of the GP $\tilde{U}$ are $\tilde{m}$ and $\tilde{k}$ (see equations \eqref{eq:krig mean} and \eqref{eq:krig cov}).
\end{proof}
Proposition \ref{prop : inheritance pde conditioned} shows that when $U$ is a GP, the results of Proposition \ref{prop :diff_constraints_distrib} are inherited on the conditioned posterior process $\tilde{U}$. One weak consequence of Proposition \ref{prop : inheritance pde conditioned} is that if GPR is performed with a covariance function $k$ that verifies point $(ii)$ of Proposition \ref{prop :diff_constraints_distrib}, then all the possible Kriging means provided by GPR remain solutions of the PDE $L(\tilde{m}) = 0$.

\section{Gaussian processes and the 3 dimensional wave equation}\label{Section:GP_wave}
The formalism we used in the previous section is necessary to tackle hyperbolic PDEs as in some cases, their solutions only verify the PDE in a weaker sense, e.g. the distributional sense (\cite{evans1998}, Sections 2.1.1 and 7.2). Hyperbolic PDEs are typically encountered when describing finite speed propagation phenomena and their prototype is the wave equation (see equation \eqref{eq:wave_eq}); this equation is central in a number of fields such as acoustics, electromagnetics and quantum mechanics.
In this section, we derive a GP model for the solutions of the homogeneous 3D wave equation, with explicit covariance formulas in the form of convolutions. 

We show on one example that the model we obtain below is capable of dealing with an initial speed $v_0$ that is piecewise continuous and an initial position $u_0$ that has piecewise continuous derivatives, when the initial discontinuity surfaces are ``nice enough''. This is an advantage with reference to the previous models, where the sample paths actually had to be sufficiently differentiable to obtain sample path degeneracy with reference to the PDE.
\subsection{General solution to the 3 dimensional wave equation}\label{sub:sol_3D_wave}
Denote the 3D Laplace operator $\Delta = \partial_{xx}^2 + \partial_{yy}^2 + \partial_{zz}^2$ and the d'Alembert operator $\Box = 1/c^2 \partial_{tt}^2 - \Delta$ with constant wave speed $c > 0$.
We focus on the general initial value problem in the free space $\mathbb{R}^3$
\begin{align}\label{eq:wave_eq}
\begin{cases}
    \hfil \Box w &= 0 \hspace{35pt} \forall (x,t) \in \mathbb{R}^3 \times \mathbb{R}_+^*,  \\
    \hfil w(x,0) &= u_0(x) \hspace{16pt} \forall x \in \mathbb{R}^3,  \\
    (\partial_t w)(x,0) &= v_0(x) \hspace{17pt} \forall x \in \mathbb{R}^3.
\end{cases}
\end{align}
Throughout this article, we will refer to $u_0$ as the initial position and $v_0$ as the initial speed. The solution of this problem is unique in the distributional sense (\cite{Duistermaat2010}, p. 164). It can be extended to all $t\in \mathbb{R}$ (\cite{Duistermaat2010}, p. 295) and is represented as follow (\cite{Duistermaat2010}, p. 295 again)
\begin{align}\label{eq sol wave}
    w(x,t) = (F_t * v_0)(x) + (\Dot{F}_t * u_0)(x) \hspace{10mm} \forall (x,t) \in \mathbb{R}^3 \times \mathbb{R},
\end{align}
where $F_t$ and $\Dot{F}_t$ are known generalized functions. That is, the function $w(x,t)$ above is a solution of the system \eqref{eq:wave_eq}, in which $t$ is now allowed to lie in $\mathbb{R}$ rather than $\mathbb{R}_+^*$. The existence of such an extension is possible because of the time reversibility of the wave equation (in the language of semigroup theory, its semigroup can be embedded in a group, \cite{pazy_sg}, Theorem 4.5 p. 222). In dimension 3, $F_t$ and $\Dot{F}_t$ are compactly supported generalized functions of order $0$ and $1$ respectively. They are given by
\begin{align}\label{eq:ft ftp in 3D}
    F_t = \frac{\sigma_{c|t|}}{4\pi c^2 t} \ \ \ \text{ and } \ \ \ \Dot{F}_t = \partial_t F_t \ \ \ \forall t \in \mathbb{R},
\end{align}
where $\sigma_R$ is the surface measure of the sphere of center $0$ and radius $R$; $\Dot{F}_t = \partial_t F_t$ means that for all $f \in C_c^1(\mathbb{R}^3), \langle\Dot{F}_t,f\rangle = \partial_t \langle F_t,f\rangle$. We make these expressions more explicit in equation \eqref{eq:kirschoff}, using spherical coordinates. It is worth noting that $(F_t)_{t\in\mathbb{R}}$ corresponds to the \textit{Green's function} of the wave equation, in the sense that it verifies the system \eqref{eq:wave_eq} with $u_0 = 0$ and $v_0 = \delta_0$ where $\delta_0$ is the Dirac mass (\cite{Duistermaat2010}, pp. 294-295). As discussed in Remark \ref{rk:meas_val_sol}, $(F_t)_{t\in\mathbb{R}}$ is a family of singular measures and this PDE system has to be understood in the distributional sense. Note also that equations \eqref{eq:ft ftp in 3D} show that $F_t$ and $\dot{F}_t$ are supported on the sphere of radius $c|t|$: ``the support of $F_t$ propagates at finite speed $c$''. This property is known as the Huygens principle for the  three dimensional wave equation, see \cite{evans1998}, p. 80.

Suppose that $u_0 \in C^1(\mathbb{R}^3) $ and $v_0 \in C^0(\mathbb{R}^3)$, then $w$ as defined in equation \eqref{eq sol wave} is a pointwise defined function (Section \ref{subsub:conv}) and in that case an explicit formula for such convolutions is reminded in equation \eqref{eq:conv distrib} (yet one may actually make sense out of \eqref{eq sol wave} when $u_0$ and $v_0$ are only required to be any generalized functions, see \cite{treves2006topological}, Chapter 27). 

Equation \eqref{eq sol wave} can be written using means over spheres. Denote $(r,\theta,\phi), r\geq 0, \theta \in [0,\pi],  \phi \in [0,2\pi]$ the spherical coordinates, $S(0,1)$  the unit sphere of $\mathbb{R}^3$ and $\gamma = (\sin \theta \cos \phi,\sin \theta \sin \phi,  \cos \theta)^T$ the corresponding parametrization of $S(0,1)$ ($||\gamma||_2=1$). We write $d\Omega = \sin \theta d\theta d\phi$ the surface differential element of $S(0,1)$.
The formulas \eqref{eq sol wave} and \eqref{eq:ft ftp in 3D} then lead to the Kirschoff formula (\cite{evans1998}, p. 72):
\begin{align}\label{eq:kirschoff}
w(x,t) = \int_{S(0,1)}tv_0(x-c|t|\gamma) + u_0(x-c|t|\gamma) -  c|t|\gamma \cdot \nabla u_0(x-c|t|\gamma) \frac{d\Omega}{4\pi}
\end{align}

\subsection{Gaussian process modelling of the solution}\label{sub:gp_model_wave}
Suppose now that $u_0$ and $v_0$ are unknown, and only pointwise values of $w$ are observed. We thus model $u_0$ and $v_0$ as random functions and put Gaussian process priors over $u_0$ and $v_0$. More precisely, we make the following assumptions.
\begin{enumerate} [label=(\subscript{A}{{\arabic*}})]
    \item \label{assump1} Suppose that the initial conditions $u_0$ and $v_0$ of Problem \eqref{eq:wave_eq} are sample paths drawn from two independent Gaussian processes $U^0 \sim GP(0,k_{\mathrm{u}})$ and $V^0 \sim GP(0,k_{\mathrm{v}})$: $\exists \omega \in \Omega, \forall x \in \mathbb{R}^3, u_0(x) = U^0_{\omega}(x)$ and $v_0(x) = V^0_{\omega}(x)$.
    \item \label{assump2} Suppose that all sample paths of $U^0$ lie in $C^1(\mathbb{R}^3)$ and that those of $V^0$ lie in $C^0(\mathbb{R}^3)$, almost surely. A sufficient condition for this is given in \cite{adler2007}, Theorem 1.4.2. This theorem states that under mild technical assumptions, the paths of $(U(x))_{x \in \calD} \sim GP(0,k)$ lie in $C^l$ a.s. as soon as $ k \in C^{2l}(\calD \times \calD)$, \textit{which we assume from now on}. This is e.g. fulfilled by the Matérn covariance functions from equation \eqref{eq:matern}, with $l=0$ for $k_{1/2}$ and $l=1$ for $k_{3/2}$.
\end{enumerate}
We now analyse the consequence of these two assumptions. First, they imply that by solving \eqref{eq:wave_eq}, one obtains a time-space stochastic process $W(x,t)$ defined by 
\begin{align}\label{eq:W sol}
    W(x,t) : \Omega \ni \omega \longmapsto (F_t * V^0_{\omega})(x) +  (\Dot{F}_t * U^0_{\omega})(x).    
\end{align}
Here again, $V^0_{\omega}$ denotes the sample path of $V^0$ at $\omega \in \Omega$ and likewise for $U_{\omega}^0$. In particular, thanks to assumption $(A_2)$, equation \eqref{eq:W sol} defines a random variable for all $(x,t)$. 
Note the space-time variable $z = (x,t)$ and note the random variables
\begin{align}\label{eq:U_and_V}
V(z) : \omega \longmapsto (F_t * V^0_{\omega})(x) \ \text{ and } \ U(z) :  \omega \longmapsto (\Dot{F}_t * U^0_{\omega})(x),  
\end{align}
that is, $W(z) = U(z) + V(z)$. We show in the next proposition that the random fields $U,V$ and $W$ are GPs as well. In particular we describe their covariance functions.
\begin{proposition}\label{prop : wave kernel}Define the two functions
\begin{align}
{k_{\mathrm{v}}^{\mathrm{wave}}(z,z') =[(F_t \otimes F_{t'}) * k_{\mathrm{v}}](x,x')},\label{eq: kv wave} \\
{k_{\mathrm{u}}^{\mathrm{wave}}(z,z') = [(\Dot{F}_t \otimes \Dot{F}_{t'}) * k_{\mathrm{u}}](x,x')}. \label{eq: ku wave}
\end{align}
(i) Then $U = (U(z))_{z \in \mathbb{R}^3 \times \mathbb{R}}$ and $V =(V(z))_{z \in \mathbb{R}^3 \times \mathbb{R}}$ as defined in \eqref{eq:U_and_V} are two independent centered GPs with covariance functions $k_{\mathrm{u}}^{\mathrm{wave}}$ and $k_{\mathrm{v}}^{\mathrm{wave}}$ respectively. Consequently, $(W(z))_{z \in \mathbb{R}^3 \times \mathbb{R}}$ is a centered GP whose covariance function is given by
\begin{align}\label{eq:wave kernel}
k_{W}(z,z') = k_{\mathrm{v}}^{\mathrm{wave}}(z,z') + k_{\mathrm{u}}^{\mathrm{wave}}(z,z').
\end{align}
(ii) Conversely, any measurable centered second order random field with covariance function $k_{W}$ has its sample paths solution of the wave equation \eqref{eq:wave_eq}, in the sense of distributions, almost surely.
\end{proposition}
The formulas \eqref{eq: kv wave} and \eqref{eq: ku wave} can easily be derived formally, by running computations as if $F_t$ and $\Dot{F}_t$ were regular generalized functions (Section \ref{subsub:regular_distrib}). This is somewhat justified because any generalized function can be approximated with a sequence of smooth compactly supported functions, by a ``cutting and regularizing'' argument (\cite{treves2006topological}, Theorem 28.2, Chapter 28). However, checking that this procedure passes to the limit everywhere is tedious. Here, we rather make use of representations of $F_t$ and $\Dot{F}_t$ thanks to Radon measures (Sections \ref{subsub:radon} and \ref{subsub:finite_order}) and use Fubini's theorem. We refer to Sections \ref{subsub:finite_order} and \ref{subsub:tensor_gen_func} for the definition of $\Dot{F}_t \otimes \Dot{F}_{t'}$, and Section \ref{subsub:conv} for the definition of $(\Dot{F}_t \otimes \Dot{F}_{t'}) * k_{\mathrm{u}}$.
\begin{proof}
$(i)$ : first we prove that $U$ and $V$ are GPs. Since $U^0$ and $V^0$ are GPs, $\mathcal{L}(U^0)$ and $\mathcal{L}(V^0)$ are only comprised of Gaussian random variables (see Section \ref{subsub:second_GP}). We then rely on the Kirschoff formula \eqref{eq:kirschoff}, writing the integrals as limits of Riemann sums.
We start with $V$, that is, we focus on the first term in Kirschoff's formula \eqref{eq:kirschoff}. To show that $V$ is a Gaussian process, we only need to show that for any $z$, $V(z) \in \mathcal{L}(V^0)$ as this will ensure the Gaussian process property. Since the sample paths of $V^0$ are continuous almost surely, there exists a sequence of numbers $(a_{k}^n)\subset\mathbb{R}$ and points $(y_k^n)\subset S(0,1)$ such that for almost any $\omega \in \Omega$,
\begin{align*}
V(z)(\omega) &= (F_t * V_{\omega}^0)(x) = t\int_{S(0,1)} V^0(x-c|t|\gamma)(\omega)\frac{d\Omega}{4\pi} \\
= &\frac{t}{4\pi}\int_{0}^{2\pi}\int_{0}^{\pi} V^0(x-c|t|\gamma(\theta,\phi))(\omega) \sin(\theta)d\theta d\phi =  \lim_{n \rightarrow \infty} \sum_{k=1}^n a_{k}^n V^0(x-c|t|y_k^n)(\omega).
\end{align*}
This shows that $V(z)$ is the a.s. limit of the sequence of centered Gaussian random variables $(Y_n)\subset \mathcal{L}(V^0)$, where $Y_n = \sum_{k=1}^n a_{k}^n V^0(x-c|t|y_k^n)$; $Y_n$ is Gaussian because $V^0$ is a GP. Almost sure convergence implies convergence in law. From \cite{legall2013}, Proposition 1.1, $V(z)$ is normally distributed and the convergence also takes place in $L^2(\mathbb{P})$. Therefore, $V(z) \in \mathcal{L}(V^0)$ and $V$ is a Gaussian process. From the same proposition, $V(z)$ is centered because the variables $Y_n$ are centered. Note that since $F_t$ is supported on the compact set $S(0,c|t|)$, we only required the sample paths of $V^0$ to be continuous rather than continuous and compactly supported.

We apply the same reasoning to $U$, by applying the above steps to the second part of Kirschoff's formula \eqref{eq:kirschoff}. One's ability to write out the integrals as a limit of Riemann sums is ensured when the sample paths of $U^0$ lie in $C^1(\mathbb{R}^3)$.

Finally, since $U^0$ and $V^0$ are independent, $\mathcal{L}(U^0)$ and $\mathcal{L}(V^0)$ are orthogonal in $L^2(\mathbb{P})$. Since $\mathcal{L}(U) \subset \mathcal{L}(U^0)$ and likewise for $V$, $U$ and $V$ are independent Gaussian processes as for Gaussian random variables, independence is equivalent to null covariance. Finally, the sum of independent Gaussian random variables is a Gaussian random variable. Therefore $\mathcal{L}(W) \subset \mathcal{L}(U) + \mathcal{L}(V)$ is only comprised of Gaussian random variables and $W$ is a Gaussian process. Now, we prove that
\begin{align}\label{eq:ftp ftp fubini}
\mathbb{E}[U(z)U(z')] = [(\dot{F}_t \otimes \dot{F}_{t'})*k_{\mathrm{u}}](x,x').
\end{align}
The main argument is Fubini's theorem for Radon measures. For this we use the fact that $\dot{F}_t$ is a distribution of order $1$ and can be identified to a sum of derivatives of measures (see equation \eqref{eq:finite order radon}): for all $t \in \mathbb{R}$, there exists $\{\mu^{t}_{\alpha}\}_{\alpha \in \mathbb{N}^3,|\alpha|\leq 1}$ a family of Radon measures such that
\begin{align}
    \Dot{F}_t &= \sum_{|\alpha| \leq 1} \partial^{\alpha} \mu^{t}_{\alpha} \ \ \ \text{in the sense of distributions}.
\end{align}
Moreover, $\dot{F}_t$ is compactly supported, therefore all the measures $\mu^{t}_{\alpha}$ are also compactly supported. 
First, we write $U_{\omega}(z)$ in integral form:
\begin{align}
U_{\omega}(z) &= \big(\Dot{F}_t * U_{\omega}^0\big)(x) = \langle \Dot{F}_t, \tau_{-x}\check{U}_{\omega}^0 \rangle = \Big \langle \sum_{|\alpha| \leq 1} \partial^{\alpha} \mu^{t}_{\alpha}, \tau_{-x}\check{U}_{\omega}^0 \Big \rangle \\
    &= \sum_{|\alpha|\leq 1} \langle \mu^{t}_{\alpha}, (-1)^{|\alpha|}\partial^{\alpha} \tau_{-x}\check{U}_{\omega}^0 \rangle = \sum_{|\alpha|\leq 1} \int_{\mathbb{R}^3} (-1)^{|\alpha|}\partial^{\alpha} U_{\omega}^0(x-y) \mu^{t}_{\alpha}(dy). 
\end{align}
Before applying Fubini's theorem, we need to check an integrability condition. Let $\alpha \in \mathbb{N}^3$ be such that $|\alpha| \leq 1$. Recall that $|\mu_t^{\alpha}|$ is defined in Section \ref{subsub:radon}; denote also $\sigma_{\partial^{\alpha} U^0}(x) = \sqrt{\text{Var}(\partial^{\alpha} U_0(x))}$. Since the sample paths of $U^0$ lie in $C^1(\calD)$ a.s, those of $\partial^{\alpha} U^0$ lie in $C^0(\calD)$ and thus the function $x \longmapsto \text{Var}(\partial^{\alpha} U_0(x))$ also lies in $C^0(\calD)$ (\cite{azais_level_2009}, chapter 1, Section 4.3). Therefore the function $x \longmapsto \sigma_{\partial^{\alpha} U^0}(x)$ also lies in $C^0(\calD)$. We now check that the integral $I$ below is finite. We use Tonelli's theorem and the Cauchy-Schwarz inequality:
\begin{align*}
    I :=& \int_{\Omega} \sum_{|\alpha|\leq 1}\int_{\mathbb{R}^3}\Big|\partial^{\alpha} U_{\omega}^0(x-y) \Big| |\mu^{t}_{\alpha}|(dy) \sum_{|\alpha'|\leq 1} \int_{\mathbb{R}^3} \Big|\partial^{\alpha'} U_{\omega}^0(x'-y') \Big| |\mu^{t'}_{\alpha'}|(dy')  \mathbb{P}(d\omega) \\
    =& \sum_{|\alpha|,|\alpha'|\leq 1} \int_{\mathbb{R}^3}\int_{\mathbb{R}^3} \int_{\Omega}  \Big|\partial^{\alpha} U_{\omega}^0(x-y)\partial^{\alpha'} U_{\omega}^0(x'-y')\Big| \mathbb{P}(d\omega) |\mu^{t}_{\alpha}|(dy)|\mu^{t'}_{\alpha'}|(dy') \\
    =& \sum_{|\alpha|,|\alpha'|\leq 1} \int_{\mathbb{R}^3}\int_{\mathbb{R}^3} \mathbb{E}\big[|\partial^{\alpha}U^0(x-y)\partial^{\alpha'}U^0(x'-y')|\big] |\mu^{t}_{\alpha}|(dy)|\mu^{t'}_{\alpha'}|(dy') \\
    \leq& \sum_{|\alpha|,|\alpha'|\leq 1} \int_{\mathbb{R}^3}\int_{\mathbb{R}^3} \bigg(\mathbb{E}\big[\partial^{\alpha}U^0(x-y)^2\big]\mathbb{E}\big[\partial^{\alpha'}U^0(x'-y')^2\big]\bigg)^{1/2}|\mu^{t}_{\alpha}|(dy)|\mu^{t'}_{\alpha'}|(dy') \\
    \leq& \Bigg( \sum_{|\alpha|\leq 1} \int_{\mathbb{R}^3} \bigg(\mathbb{E}\big[\partial^{\alpha}U^0(x-y)^2\big]\bigg)^{1/2} |\mu^{t}_{\alpha}|(dy) \Bigg)\times \Bigg( \sum_{|\alpha|\leq 1} \int_{\mathbb{R}^3} \bigg(\mathbb{E}\big[\partial^{\alpha}U^0(x-y)^2\big]\bigg)^{1/2} |\mu^{t'}_{\alpha}|(dy) \Bigg) 
    \\ \leq &\Big( \sum_{|\alpha|\leq 1} (|\mu^{t}_{\alpha}| * \sigma_{\partial^{\alpha} U^0})(x) \Big) \times \Big( \sum_{|\alpha|\leq 1} (|\mu^{t'}_{\alpha}| * \sigma_{\partial^{\alpha} U^0})(x') \Big) < +\infty.
\end{align*}
For all multi-index $\alpha$, the scalar $(|\mu^{t}_{\alpha}| * \sigma_{\partial^{\alpha} U^0})(x)$ is finite because $x \longmapsto \sigma_{\partial^{\alpha} U^0}(x)$ is continuous and $|\mu^{t}_{\alpha}|$ is compactly supported.
Note also that from Assumption $(A_2)$, the GP $U^0$ is \textit{mean square differentiable} up to order 1, which implies (\cite{ritter2007average}, Section III.1.4) that we have, for all multi-indexes $\alpha,\alpha'$ such that $|\alpha|,|\alpha'| \leq 1$, $x$ and $x'$:
\begin{align}
\mathbb{E}\big[\partial^{\alpha} U^0(x)\partial^{\alpha'}U^0(x')\big] = \partial_1^{\alpha} \partial_2^{\alpha'}k_{\mathrm{u}}(x,x').
\end{align}
where $\partial_1$ (respectively  $\partial_2$) denotes derivatives with reference to the first (respectively second) argument of $k_{\mathrm{u}}$.
We may thus permute integrals and differential operators in $\mathbb{E}\big[U(z)U(z')\big]$: 
\begin{align}
    \mathbb{E}\big[U(z)U(z')\big] &= \mathbb{E}\Bigg[\sum_{|\alpha|\leq 1} \int_{\mathbb{R}^3} (-1)^{|\alpha|}\partial^{\alpha} U^0(x-y) \mu^{t}_{\alpha}(dy))\sum_{|\alpha'|\leq 1} \int_{\mathbb{R}^3} (-1)^{|\alpha'|}\partial^{\alpha'} U^0(x-y) \mu^{t'}_{\alpha'}(dy')\Bigg] \nonumber \\
    &= \sum_{|\alpha|,|\alpha'| \leq 1}\int_{\mathbb{R}^3} \int_{\mathbb{R}^3}(-1)^{|\alpha|}(-1)^{|\alpha'|}\partial_1^{\alpha} \partial_2^{\alpha'}\mathbb{E}\big[U^0(x-y)U^0(x'-y')\big]\mu^{t}_{\alpha}(dy)\mu^{t'}_{\alpha'}(dy') \nonumber \\
    &= \sum_{|\alpha|,|\alpha'| \leq 1}\int_{\mathbb{R}^3} \int_{\mathbb{R}^3}(-1)^{|\alpha|}(-1)^{|\alpha'|}\partial_1^{\alpha} \partial_2^{\alpha'}k_{\mathrm{u}}(x-y,x'-y')\mu^{t}_{\alpha}(dy)\mu^{t'}_{\alpha'}(dy') \nonumber \\
    &= \bigg[\Big(\sum_{|\alpha|\leq 1} \partial^{\alpha}\mu^{t}_{\alpha}  \otimes \sum_{|\alpha'| \leq 1} \partial^{\alpha'}\mu^{t'}_{\alpha'}\Big) * k_{\mathrm{u}}\bigg](x,x') = [(\Dot{F}_t \otimes \dot{F}_{t'}) * k_{\mathrm{u}}](x,x'), \nonumber
\end{align}
which proves \eqref{eq:ftp ftp fubini}. 

One proves that $\mathbb{E}\big[V(z)V(z')\big] = [({F}_t \otimes {F}_{t'}) * k_{\mathrm{v}}](x,x')$ the exact same way, which is actually simpler as $F_t$ is directly a measure. To conclude,
\begin{align}
    k_{W}(z,z') &= \text{Cov}(W(z),W(z')) \nonumber \\
    &= \mathbb{E}[(W(z)W(z')] = \mathbb{E}\big[\big(U(z) + V(z)\big)\big(U(z') + V(z') \big)\big] \nonumber\\
    &= \mathbb{E}\big[U(z)U(z')\big] + \mathbb{E}\big[U(z)V(z')\big] +  \mathbb{E}\big[V(z)U(z')\big] +  \mathbb{E}\big[V(z)V(z')\big] \nonumber \\
    &= [(\Dot{F}_t \otimes \Dot{F}_{t'}) * k_{\mathrm{u}}](x,x') + [(F_t \otimes F_{t'}) * k_{\mathrm{v}}](x,x').
\end{align}
The cross terms are null because $U(z)$ and $V(z')$ are independent as well as $U(z')$ and $V(z)$.

$(ii)$ : with expression \eqref{eq:wave kernel}, one checks that for any fixed $z'$, the function $z \longmapsto k_{W}(z,z')$ is of the form \eqref{eq sol wave} and thus verifies $\Box k_{x'} = 0$ in the sense of distributions. $(ii)$ is then a direct consequence of Proposition \ref{prop :diff_constraints_distrib}.
\end{proof}
\begin{remark}
If $U$ and $V$ are not independent, then the two terms $[(\Dot{F}_t \otimes F_{t'}) * k_{uv}](x,x')$ and $[(F_t \otimes \Dot{F}_{t'}) * k_{vu}](x,x')$ must be added to equation \eqref{eq:wave kernel}, where $k_{uv}(x,x')$ denotes the cross covariance between $U$ and $V$ : $k_{uv}(x,x') = \text{Cov}(U(x),V(x'))$ and $k_{vu}(x,x') = \text{Cov}(V(x),U(x')) = k_{uv}(x',x)$.
\end{remark}
More explicitly, we have the following Kirschoff-like integral formulas for $k_{\mathrm{v}}^{\mathrm{wave}}$ and $k_{\mathrm{u}}^{\mathrm{wave}}$:
\begin{align}
[(F_t \otimes F_{t'}) * k_{\mathrm{v}}](x,x') &= tt'\int_{S(0,1) \times S(0,1)}k_{v}(x-c|t|\gamma,x'-c|t'|\gamma')\frac{d\Omega d\Omega'}{(4\pi)^2}, \label{eq:explicit Ft Ftp}\\
[(\Dot{F}_t \otimes \Dot{F}_{t'}) * k_{\mathrm{u}}](x,x') &= \int_{S(0,1) \times S(0,1)}\Big( k_{u}(x-c|t|\gamma,x'-c|t'|\gamma') \nonumber \\ 
& \hspace{30pt}-c|t|\nabla_1 k_{u}(x-c|t|\gamma,x'-c|t'|\gamma')\cdot \gamma \nonumber \\
& \hspace{30pt}-c|t'| \nabla_2 k_{u}(x-c|t|\gamma,x'-c|t'|\gamma')\cdot \gamma' \nonumber\\
&\hspace{30pt}+ c^2tt' \gamma^T \nabla_1 \nabla_2 k_{u}(x-c|t|\gamma,x'-c|t'|\gamma')\gamma'\Big) \frac{d\Omega d\Omega'}{(4\pi)^2}.\label{eq:explicit Ft Ftp point}
\end{align}
Above, $\nabla_1 k_{u}(x,x')$ is the gradient vector of $k_{u}$ with reference to $x$,  $\nabla_2 k_{u}(x,x')$ is the gradient vector of $k_{u}$ with reference to $x'$ and $\nabla_1 \nabla_2 k_{u}(x,x')$ is the matrix whose entry $(i,j)$ is given by 
\begin{align}
\nabla_1 \nabla_2 k_{u}(x,x')_{ij} = \partial_{x_i^1}\partial_{x_j^2}k_{u}(x,x').
\end{align}
($\partial_{x_i^1}$ is the derivative with reference to the $i^{th}$ coordinate of $x$, $\partial_{x_j^2}$ is the derivative with reference to the $j^{th}$ coordinate of $x'$).

\subsubsection{\texorpdfstring{Extending the covariance functions $k_{\mathrm{u}}^{\mathrm{wave}}$ and $k_{\mathrm{v}}^{\mathrm{wave}}$ to initial conditions $u_0$ and $v_0$ with piecewise regularity}{extendingcov}}\label{subsub:extend_kwave}

The formulas \eqref{eq:explicit Ft Ftp} and \eqref{eq:explicit Ft Ftp point} are valid in a more general context than that of assumptions \ref{assump1} and \ref{assump2}. We provide below examples where these formulas yield valid covariance functions (in particular, functions defined for \textit{all} values of $(x,t)$ and $(x',t')$) corresponding to initial conditions with some forms of piecewise discontinuities. Assume, for example, that the initial speed $v_0$ is compactly supported on a ball $B(x_0,R)$ centered on some point $x_0$ with radius $R$. This is a natural model when $v_0$ is assumed to be a localized source. For the process $V^0$, this translates as $V^0(x) = 0\ a.s.$ if $x$ is outside the ball $B(x_0,R)$. One can thus truncate the covariance function of $V^0$ accordingly, e.g. choosing the following function for $k_{\mathrm{v}}$ (see Section \ref{subsub:second_GP} for $k_{1/2}$)
\begin{align}
k_{\mathrm{v}}(x,x') = k_{1/2}(x,x')\mathbbm{1}_{[0,R]}(||x-x_0||)\mathbbm{1}_{[0,R]}(||x'-x_0||).
\end{align}
Above, $||x||$ denotes the Euclidean norm of $x$. Such a covariance function indeed verifies $k_{\mathrm{v}}(x,x) = \text{Var}(V^0)(x))= 0$ if $||x-x_0|| > R$ and the GP corresponding to $k_{\mathrm{v}}$ is $V^0(x) = V_{1/2}(x)\mathbbm{1}_{[0,R]}(||x-x_0||)$, where $V_{1/2}$ is a continuous modification of a GP with covariance function $k_{1/2}$. Note that the sample paths of $V^0$ are piecewise continuous and that $V$ as defined in \eqref{eq:U_and_V} is well-defined and measurable.
The integrals in \eqref{eq:explicit Ft Ftp} still make sense and point $(ii)$ from Proposition is still valid: the sample paths of the process $V$ whose covariance function is $k_{\mathrm{v}}^{\mathrm{wave}}$ (or any other measurable centered second order random field with this covariance function) remains a solution of the wave equation in the distributional sense. One can perform the same kind of discussions on $k_{\mathrm{u}}^{\mathrm{wave}}$: for example, equation \eqref{eq:explicit Ft Ftp point} shows that when ${k_{\mathrm{u}} \in C^{1,1}(\mathbb{R}^3\times \mathbb{R}^3)\setminus C^{2,2}(\mathbb{R}^3\times \mathbb{R}^3)}$, $k_{\mathrm{u}}^{\mathrm{wave}}$ is only expected to lie in $C^{1,1}(\mathbb{R}^3\times \mathbb{R}^3)$; the sample paths of the GP with covariance function $k_{\mathrm{u}}^{\mathrm{wave}}$ will be at most of class $C^1$ and thus cannot be strong solutions of equation \eqref{eq:wave_eq}. This is the case when $k_{\mathrm{u}}$ is the $k_{3/2}$ Matérn covariance function from equation \eqref{eq:matern}.

More generally, one can incorporate a finite number of discontinuities on $k_{\mathrm{v}}$ and on the derivatives of $k_{\mathrm{u}}$ so that they remain piecewise continuous: the integrals above will remain well defined and the sample paths of the corresponding GPs will remain distributional solutions to the wave equation, even though they will not be sufficiently differentiable to be strong solutions.
\section{Conclusion and perspectives}
In Section \ref{Section:sto diff}, we have presented a new result that provides a simple characterization of the measurable second order random fields $(U(x))_{x\in\calD}$ whose sample paths verify homogeneous linear differential constraints within the framework of generalized functions. This characterization is valid for any linear differential operator $L$, provided that its coefficients fulfil minimal smoothness requirements, and no stationarity assumptions over $(U(x))_{x\in\calD}$ are required. 
Motivated by physical applications,
we described in Section \ref{Section:GP_wave} a Gaussian process model of the wave equation which is central to describe propagation phenomena.  This PDE served as an application case for Proposition \ref{prop :diff_constraints_distrib}, and the GP model was derived by putting a GP prior on the wave equation's initial conditions. In Proposition \ref{prop : wave kernel}, we presented covariance formulas that are tailored to the wave equation and take the form of convolutions; these expressions are interesting in themselves and call for physics-informed GPR applications for this equation. In particular, we showed that these formulas can model piecewise continuously differentiable solutions for the wave equation. Moreover, this setting provides a natural way to incorporate any type of information, both numerical or experimental. In a forthcoming paper, we will show how to use GPR conjointly with the covariance functions from Proposition \ref{prop : wave kernel} in a numerical setting, in order to construct approximate solutions of the wave equation based on scattered observations. This in turn provides a natural method for solving different inverse problems, by using the likelihood of GPR as well as the reconstructed solution. An other application concerns the design of transparent boundary conditions (TBC): this provides artificial boundary conditions on a computational domain so that the computed solution is exactly an approximation of a solution on the whole space. Those conditions are usually nonlocal and restricted to simple geometries. A GPR strategy is meshless therefore suitable to design TBC on any type of computational domains.

Proposition \ref{prop :diff_constraints_distrib} constitutes a first step towards understanding PDE constrained random fields in an weakened sense; different functional analysis frameworks can now be considered, obvious extensions being the weak or variational formulations of equation \eqref{eq:edp}. These formulations are obtained by transferring only a part of the derivatives of the PDE to the test function and are for instance the canonical way of studying elliptic PDEs (\cite{evans1998}, Section 6.1.2). The natural spaces arising from these formulations are Sobolev spaces rather than $\scrD'(\calD)$. An attached question, as studied in \cite{SCHEUERER2010}, is that of the Sobolev regularity of a given second order random field; a current research topic is whether or not one may relax the continuity assumptions required in \cite{SCHEUERER2010}. Finally, the matter of using random fields for modelling and approximating solutions of nonlinear PDEs is a natural direction for future research.

\subsubsection*{Acknowledgements} Research of all the authors was supported by SHOM (Service Hydrographique et Océanographique de la Marine) project ``Machine Learning Methods in Oceanography'' no-20CP07. We thank Rémy Baraille in particular for his personal involvement in the project. We are thankful to the reviewers and the editor for their interesting and constructive remarks, leading to an improved and enriched version of the manuscript.

\bibliographystyle{abbrv} 
\bibliography{bibliography_bernoulli}

\begin{thebibliography}{10}

\bibitem{adler1981}
R.~J. Adler.
\newblock {\em The Geometry of Random Fields}.
\newblock Classics in Applied Mathematics. Philadelphia, PA: SIAM, 2010.

\bibitem{adler2007}
R.~J. Adler and J.~E. Taylor.
\newblock {\em Random Fields and Geometry}.
\newblock Springer Monographs in Mathematics. New York, NY: Springer, 2007.

\bibitem{albert2020}
C.~G. Albert and K.~Rath.
\newblock {G}aussian process regression for data fulfilling linear differential
  equations with localized sources.
\newblock {\em Entropy}, 22(2), 2020.

\bibitem{alvarez2013}
M.~{\'A}lvarez, D.~Luengo, and N.~Lawrence.
\newblock Linear latent force models using {G}aussian processes.
\newblock {\em IEEE Trans. Pattern Anal. Mach. Learn. Intell.}, 35:2693--2705,
  2013.

\bibitem{azais_level_2009}
J.-M. Aza{\"i}s and M.~Wschebor.
\newblock {\em {Level sets and extrema of random processes and fields}}.
\newblock Hoboken, NJ: {Wiley \& Sons, Inc.}, 2009.

\bibitem{agnan2004}
A.~Berlinet and C.~Thomas-Agnan.
\newblock {\em Reproducing Kernel {H}ilbert Spaces in Probability and
  Statistics}.
\newblock New York, NY: Springer, 2004.

\bibitem{brezis2010functional}
H.~Brezis.
\newblock {\em Functional Analysis, Sobolev Spaces and Partial Differential
  Equations}.
\newblock Universitext. New York, NY: Springer, 2010.

\bibitem{vergara2022general}
R.~Carrizo-Vergara, D.~Allard, and N.~Desassis.
\newblock A general framework for {SPDE}-based stationary random fields.
\newblock {\em Bernoulli}, 28(1):1--32, 2022.

\bibitem{CHEN_owhadi_2021}
Y.~Chen, B.~Hosseini, H.~Owhadi, and A.~M. Stuart.
\newblock Solving and learning nonlinear {PDE}s with {G}aussian processes.
\newblock {\em J. Comput. Phys.}, 447:110668, 2021.

\bibitem{da2014stochastic}
G.~Da~Prato and J.~Zabczyk.
\newblock {\em Stochastic equations in infinite dimensions}, volume~44 of {\em
  Encyclopedia of Mathematics and its Applications}.
\newblock Cambridge: Cambridge University Press, 1992.

\bibitem{dieudonne_separable}
J.~{Dieudonn\'e}.
\newblock {Sur les espaces de Montel m\'etrisables}.
\newblock {\em {C. R. Acad. Sci., Paris}}, 238:194--195, 1954.

\bibitem{doob1937stochastic}
J.~L. Doob.
\newblock Stochastic processes depending on a continuous parameter.
\newblock {\em Trans. Amer. Math. Soc.}, 42(1):107--140, 1937.

\bibitem{doob_sto_pro}
J.~L. Doob.
\newblock {\em Stochastic processes}.
\newblock Wiley Classics Library. New York, NY: John Wiley \& Sons, Inc., 1990.

\bibitem{Duistermaat2010}
J.~J. Duistermaat and J.~A.~C. Kolk.
\newblock {\em Distributions}.
\newblock Boston, MA: Birkh\"{a}user Boston, Inc., 2010.

\bibitem{estrade2020anisotropic}
A.~Estrade and J.~Fournier.
\newblock Anisotropic gaussian wave models.
\newblock {\em ALEA Lat. Am. J. Probab. Math. Stat.}, 17, 2020.

\bibitem{evans1998}
L.~C. Evans.
\newblock {\em Partial differential equations}, volume~19 of {\em Graduate
  Studies in Mathematics}.
\newblock Providence, RI: American Mathematical Society, second edition, 2010.

\bibitem{evans2018measure}
L.~C. Evans and R.~F. Gariepy.
\newblock {\em Measure theory and fine properties of functions}.
\newblock Textbooks in Mathematics. Boca Raton, FL: CRC Press, revised edition,
  2015.

\bibitem{fan2018modeling}
M.~Fan, D.~Paul, T.~C. Lee, and T.~Matsuo.
\newblock Modeling tangential vector fields on a sphere.
\newblock {\em J. Amer. Statist. Assoc.}, 113(524):1625--1636, 2018.

\bibitem{fasshauer2007}
G.~E. Fasshauer.
\newblock {\em Meshfree approximation methods with {MATLAB}}, volume~6 of {\em
  Interdisciplinary Mathematical Sciences}.
\newblock Hackensack, NJ: World Scientific Publishing Co. Pte. Ltd., 2007.

\bibitem{fiedler2016distances}
J.~Fiedler.
\newblock {\em Distances, Gegenbauer expansions, curls, and dimples: On
  dependence measures for random fields}.
\newblock PhD thesis, Faculty of Mathematics and Computer Science, 2016.

\bibitem{fuselier2007refined}
E.~J. Fuselier, Jr.
\newblock {\em Refined error estimates for matrix-valued radial basis
  functions}.
\newblock Ann Arbor, MI: ProQuest LLC, 2006.
\newblock Thesis (Ph.D.)--Texas A\&M University.

\bibitem{gapaillard8processus}
J.~Gapaillard and J.~Michaux.
\newblock Sur les processus lin\'{e}aires d\'{e}finis sur un espace
  nucl\'{e}aire.
\newblock {\em Ann. Fac. Sci. Toulouse Math. (5)}, 8(1):75--92, 1986/87.

\bibitem{geist2020}
A.~Geist and S.~Trimpe.
\newblock Learning constrained dynamics with gauss’ principle adhering
  gaussian processes.
\newblock In {\em Proceedings of the 2nd Conference on Learning for Dynamics
  and Control}, volume 120 of {\em Proc. Mach. Learn. Res.}, pages 225--234.
  PMLR, 10--11 Jun 2020.

\bibitem{gelfand1964generalized}
I.~M. Gel'fand and N.~Y. Vilenkin.
\newblock {\em Generalized functions. {V}ol. 4: {A}pplications of harmonic
  analysis}.
\newblock New York-London: Academic Press, 1964.

\bibitem{ginsbourger2016}
D.~Ginsbourger, O.~Roustant, and N.~Durrande.
\newblock On degeneracy and invariances of random fields paths with
  applications in {G}aussian process modelling.
\newblock {\em J. Statist. Plann. Inference}, 170:117--128, 2016.

\bibitem{graepel2003}
T.~Graepel.
\newblock Solving noisy linear operator equations by gaussian processes:
  Application to ordinary and partial differential equations.
\newblock In {\em Machine Learning, Proceedings of the Twentieth International
  Conference {(ICML} 2003), August 21-24, 2003, Washington, DC, {USA}}, pages
  234--241. {AAAI} Press, 2003.

\bibitem{gulian2022}
M.~Gulian, A.~Frankel, and L.~Swiler.
\newblock {G}aussian process regression constrained by boundary value problems.
\newblock {\em Comput. Methods Appl. Mech. Engrg.}, 388:114117, 2022.

\bibitem{janson_1997}
S.~Janson.
\newblock {\em Gaussian {H}ilbert spaces}, volume 129 of {\em Cambridge Tracts
  in Mathematics}.
\newblock Cambridge: Cambridge University Press, 1997.

\bibitem{Jidling2018ProbabilisticMA}
C.~Jidling, J.~Hendriks, N.~Wahlstrom, A.~Gregg, T.~Schon, C.~Wensrich, and
  A.~Wills.
\newblock Probabilistic modelling and reconstruction of strain.
\newblock {\em Nucl. Instrum. Methods Phys. Res. B: Beam Interact. Mater. At.},
  436:141--155, 2018.

\bibitem{jidling2017}
C.~Jidling, N.~Wahlstr\"{o}m, A.~Wills, and T.~B. Sch\"{o}n.
\newblock Linearly constrained {G}aussian processes.
\newblock In {\em Adv. Neural Inf. Process Syst.}, volume~30. Curran
  Associates, Inc., 2017.

\bibitem{kuchment_tomo}
P.~Kuchment and L.~Kunyansky.
\newblock Mathematics of photoacoustic and thermoacoustic tomography.
\newblock In {\em Handbook of mathematical methods in imaging. {V}ol. 1, 2, 3},
  pages 1117--1167. New York: Springer, 2015.

\bibitem{lang1993}
S.~Lang.
\newblock {\em Real and functional analysis}, volume 142 of {\em Graduate Texts
  in Mathematics}.
\newblock New York, NY: Springer-Verlag, third edition, 1993.

\bibitem{hegerman2018}
M.~Lange-Hegermann.
\newblock Algorithmic linearly constrained {G}aussian processes.
\newblock In {\em Adv. Neural Inf. Process Syst.}, volume~31. Curran
  Associates, Inc., 2018.

\bibitem{hegermann2021LinearlyCG}
M.~Lange-Hegermann.
\newblock Linearly constrained {G}aussian processes with boundary conditions.
\newblock In {\em Proceedings of The 24th International Conference on
  Artificial Intelligence and Statistics}, volume 130 of {\em Proc. Mach.
  Learn. Res.}, pages 1090--1098. PMLR, 13--15 Apr 2021.

\bibitem{legall2013}
J.-F. Le~Gall.
\newblock {\em Mouvement brownien, martingales et calcul stochastique},
  volume~71 of {\em Math\'{e}matiques \& Applications (Berlin) [Mathematics \&
  Applications]}.
\newblock Heidelberg: Springer, 2013.

\bibitem{lindgren2022spde}
F.~Lindgren, D.~Bolin, and H.~Rue.
\newblock The spde approach for gaussian and non-gaussian fields: 10 years and
  still running.
\newblock {\em Spat. Stat.}, page 100599, 2022.

\bibitem{mendes2012}
F.~M. Mendes and E.~A. da~Costa~Júnior.
\newblock {B}ayesian inference in the numerical solution of {L}aplace's
  equation.
\newblock {\em AIP Conference Proceedings}, 1443(1):72--79, 2012.

\bibitem{Narcowich1994GeneralizedHI}
F.~J. Narcowich and J.~Ward.
\newblock Generalized {H}ermite interpolation via matrix-valued conditionally
  positive definite functions.
\newblock {\em Math. Comp.}, 63:661--687, 1994.

\bibitem{NGUYEN_peraire}
N.~Nguyen and J.~Peraire.
\newblock {G}aussian functional regression for linear partial differential
  equations.
\newblock {\em Comput. Methods Appl. Mech. Engrg.}, 287:69--89, 2015.

\bibitem{owhadi_bayes_homog}
H.~Owhadi.
\newblock {B}ayesian numerical homogenization.
\newblock {\em Multiscale Model. Simul.}, 13(3):812--828, 2015.

\bibitem{pazy_sg}
A.~Pazy.
\newblock {\em Semigroups of linear operators and applications to partial
  differential equations}, volume~44 of {\em Applied Mathematical Sciences}.
\newblock New York: Springer-Verlag, 1983.

\bibitem{raissi2017}
M.~Raissi, P.~Perdikaris, and G.~E. Karniadakis.
\newblock Machine learning of linear differential equations using {G}aussian
  processes.
\newblock {\em J. Comput. Phys.}, 348:683--693, 2017.

\bibitem{gpml2006}
C.~E. Rasmussen and C.~K.~I. Williams.
\newblock {\em Gaussian processes for machine learning}.
\newblock Adaptive Computation and Machine Learning. Cambridge, MA: MIT Press,
  2006.

\bibitem{ritter2007average}
K.~Ritter.
\newblock {\em Average-case analysis of numerical problems}, volume 1733 of
  {\em Lecture Notes in Mathematics}.
\newblock Berlin: Springer-Verlag, 2000.

\bibitem{roques2022spatial}
L.~Roques, D.~Allard, and S.~Soubeyrand.
\newblock Spatial statistics and stochastic partial differential equations: A
  mechanistic viewpoint.
\newblock {\em Spat. Stat.}, page 100591, 2022.

\bibitem{rudin1991}
W.~Rudin.
\newblock {\em Functional analysis}.
\newblock International Series in Pure and Applied Mathematics. New York, NY:
  McGraw-Hill, Inc., second edition, 1991.

\bibitem{Schaback2009SolvingTL}
R.~Schaback.
\newblock Solving the {L}aplace equation by meshless collocation using harmonic
  kernels.
\newblock {\em Adv. Comput. Math.}, 31:457--470, 2009.

\bibitem{schaefer1999}
H.~H. Schaefer and M.~P. Wolff.
\newblock {\em Topological vector spaces}, volume~3 of {\em Graduate Texts in
  Mathematics}.
\newblock New York, NY: Springer-Verlag, second edition, 1999.

\bibitem{SCHEUERER2010}
M.~Scheuerer.
\newblock Regularity of the sample paths of a general second order random
  field.
\newblock {\em Stochastic Process. Appl.}, 120(10):1879--1897, 2010.

\bibitem{scheuerer2012}
M.~Scheuerer and M.~Schlather.
\newblock Covariance models for divergence-free and curl-free random vector
  fields.
\newblock {\em Stoch. Models}, 28:433 -- 451, 2012.

\bibitem{schwartz1964sous}
L.~Schwartz.
\newblock Sous-espaces {H}ilbertiens d’espaces vectoriels topologiques et
  noyaux associ{\'e}s (noyaux reproduisants).
\newblock {\em J. Anal. Math.}, 13(1):115--256, 1964.

\bibitem{Serre1999SystemsOC}
D.~Serre.
\newblock {\em Systems of conservation laws. Vol 1.}
\newblock Cambridge: Cambridge University Press, 1999.

\bibitem{steinwart2019convergence}
I.~Steinwart.
\newblock Convergence types and rates in generic {K}arhunen-{L}oeve expansions
  with applications to sample path properties.
\newblock {\em Potential Anal.}, 51(3):361--395, 2019.

\bibitem{treves2006topological}
F.~Tr\`eves.
\newblock {\em Topological vector spaces, distributions and kernels}.
\newblock New York-London: Academic Press, 1967.

\bibitem{wahlstrom2013}
N.~Wahlstrom, M.~Kok, T.~B. Sch{\"o}n, and F.~Gustafsson.
\newblock Modeling magnetic fields using {G}aussian processes.
\newblock {\em Proc. IEEE Int. Conf. Acoust. Speech Signal Process.}, pages
  3522--3526, 2013.

\bibitem{whittle1954stationary}
P.~Whittle.
\newblock On stationary processes in the plane.
\newblock {\em Biometrika}, pages 434--449, 1954.

\end{thebibliography}
\end{document}